\documentclass[reqno]{amsart}
\usepackage{amsmath,amsthm,amscd,amssymb,amsfonts, amsbsy}
\usepackage{latexsym}
\usepackage{color}
\usepackage{enumerate}
\usepackage{pxfonts}
\usepackage{marginnote}
\usepackage{todonotes}

\theoremstyle{plain}
\newtheorem{theorem}[equation]{Theorem}
\newtheorem{lemma}[equation]{Lemma}
\newtheorem{corollary}[equation]{Corollary}
\newtheorem{proposition}[equation]{Proposition}

\theoremstyle{definition}
\newtheorem{definition}[equation]{Definition}

\newtheorem{example}[equation]{Example}

\theoremstyle{remark}
\newtheorem{remark}[equation]{Remark}

\newcommand{\dv}{\operatorname{div}}

\newcommand{\supp}{\operatorname{supp}}
\newcommand{\dist}{\operatorname{dist}}
\newcommand{\diam}{\operatorname{diam}}

\newcommand*{\tran}{^{\mkern-1.5mu\mathsf{T}}}

\numberwithin{equation}{section}

\newcommand{\bR}{\mathbb{R}}

\providecommand{\set}[1]{\{#1\}}
\providecommand{\Set}[1]{\left\{#1\right\}}

\providecommand{\abs}[1]{\lvert#1\rvert}
\providecommand{\Abs}[1]{\left\lvert#1\right\rvert}

\providecommand{\norm}[1]{\lVert#1\rVert}
\providecommand{\Norm}[1]{\left\lVert#1\right\rVert}

\renewcommand{\vec}[1]{\boldsymbol{#1}}

\DeclareMathOperator*{\esssup}{ess\,sup}
\DeclareMathOperator*{\essinf}{ess\,inf}

\begin{document}
\title[The Neumann Green function]
{The Neumann Green function and scale invariant regularity estimates for elliptic equations with Neumann data in Lipschitz domains}

\author[S. Kim]{Seick Kim}
\address[S. Kim]{Department of Mathematics, Yonsei University, 50 Yonsei-ro, Seodaemun-gu, Seoul 03722, Republic of Korea}
\email{kimseick@yonsei.ac.kr}
\thanks{S. Kim is partially supported by National Research Foundation of Korea (NRF) No. NRF-2022R1A2C1003322.}

\author[G. Sakellaris]{Georgios Sakellaris}
\address[G. Sakellaris]{Aristotle University of Thessaloniki,
School of Mathematics, 541 24 Thessaloniki, Greece}
\email{gsakell@math.auth.gr}

\subjclass[2010]{Primary 35J08, 35J25; Secondary 35B45, 35D30}

\keywords{Green function, Neumann function, Neumann boundary condition, Scale invariant estimates}

\begin{abstract}
We construct the Neumann Green function and establish scale invariant regularity estimates for solutions to the Neumann problem for the elliptic operator $Lu=-\dv(\mathbf A \nabla u + \vec b u) + \vec c \cdot \nabla u + du$ in a Lipschitz domain $\Omega$.
We assume that $\mathbf A$ is elliptic and bounded, that the lower order coefficients belong to scale invariant Lebesgue spaces, and that either $d\geq\dv\vec b$ in $\Omega$ and $\vec b\cdot\nu\geq 0$ on $\partial\Omega$ in the sense of distributions, or the analogous condition for $\vec c$ holds.
We develop the $L^2$ theory, construct the Neumann Green function and show estimates in the respective optimal spaces, and show local and global pointwise estimates for solutions.
The main novelty is that our estimates are scale invariant, since our constants depend on the lower order coefficients only via their norms, and on the Lipschitz domain only via its Lipschitz character. Moreover, our pointwise estimates are shown in the optimal scale invariant setting for the inhomogeneous terms and the Neumann data.
\end{abstract}
%\today
\maketitle

%------------------------------------------------------------------------------%
\section{Introduction}
%------------------------------------------------------------------------------%

We investigate the Neumann Green function and regularity estimates for solutions to the Neumann problem for the second-order elliptic operator
\begin{equation*}		%		\label{eq1342fri}
Lu=-\dv(\mathbf A \nabla u + \vec b u) + \vec c \cdot \nabla u + du
\end{equation*}
defined in a Lipschitz domain $\Omega \subset \bR^n$ with $n \ge 3$.
We assume that $\mathbf A=(a^{ij})$ is an $n \times n$ matrix valued function defined in $\Omega$ satisfying the uniform ellipticity and boundedness condition
\begin{equation}					\label{ellipticity}
\lambda \abs{\xi}^2 \le \mathbf A(x) \xi \cdot \xi, \quad \forall x \in \Omega,\; \forall \xi \in \bR^n, \quad\text{and}\quad \norm{\mathbf A}_{L^\infty(\Omega)} \le \Lambda,
\end{equation}
where $\lambda$ and $\Lambda$ are positive constants.
We assume that  $\vec b=(b^1,\ldots, b^n)$, $\vec c=(c^1,\ldots, c^n)$ belong to $L^n(\Omega)$ and $d \in L^{n/2}(\Omega)$.

The first goal of this article is to fully develop the $L^2$ theory for solutions to the Neumann problem for the operator $L$, also including inhomogeneous terms. Since we are interested in existence and uniqueness (up to a subspace of dimension 1), a special condition guaranteeing uniqueness of solutions to the Neumann problem should be imposed (see \eqref{eq1122wed} and \eqref{eq1123wed}). Under such a condition, we show existence, uniqueness, and scale invariant estimates for solutions and subsolutions in a scale invariant Sobolev space (see Propositions~\ref{SolvabilityDelta=0}, \ref{prop2132}, as well as Proposition~\ref{EstimateForDelta>0}). This is done by first identifying a condition that differentiates between having one-dimensional and zero-dimensional kernels to our problems (see Proposition~\ref{kernelChar}).

The second goal is the construction of the Green function $G(x,y)$ for the operator $L$ with Neumann data, which we call the Neumann Green function, as well as the scale invariant estimate $\norm{G(\cdot,y)}_{L^{\frac{n}{n-2},\infty}(\Omega)}+\norm{\nabla G(\cdot,y)}_{L^{\frac{n}{n-1},\infty}(\Omega)}\leq C$.
One of the novelties of this article in this setting is the observation that $G(\cdot, y)\in L^{\frac{n-1}{n-2},\infty}(\partial\Omega)$, in a scale invariant way (see Section~\ref{sec5}). Then, these estimates for the Neumann Green function allow us to obtain local and global pointwise estimates for subsolutions and solutions, in the optimal setting for the Neumann data and the inhomogeneous terms of the equation, for which we consider scale invariant Lorentz spaces.

The main feature of our results is the consideration of the optimal scale invariant setting for them to hold, as well as the optimal dependence of the constants on the given parameters.
That is, our constants in the estimates will depend on the lower order coefficients only via their norms, and on the Lipschitz domain only via its Lipschitz character (Definition~\ref{LipDom}), which is among the main novelties of this article.

The natural condition we will consider that guarantees existence and uniqueness up to a subspace is that the pair $(\vec b, d)$ satisfies the inequality
\begin{equation}				\label{eq1122wed}
\int_\Omega \vec b \cdot \nabla \phi + d\phi \ge 0,\quad\text{for every }\phi \in C^\infty_c(\bR^n)\text{ with }\phi \ge 0,
\end{equation}
or the adjoint operator satisfies the corresponding property; that is, the adjoint condition for \eqref{eq1122wed} is 
\begin{equation}				\label{eq1123wed}
\int_\Omega \vec c \cdot \nabla \phi + d\phi \ge 0,\quad\text{for every }\phi \in C^\infty_c(\bR^n)\text{ with }\phi \ge 0.
\end{equation}
The condition~\eqref{eq1122wed} is analogous to the condition $d \ge \dv \vec b$, which we assumed for the Dirichlet problem in \cite{KS19}.
However, unlike \cite{KS19}, here we do not assume that $\phi$ vanishes on the boundary $\partial\Omega$ and thus the condition \eqref{eq1122wed} formally becomes
\[
\int_\Omega (d-\dv \vec b)\phi + \int_{\partial \Omega} (\vec b \cdot \nu)\phi \ge 0.
\]
Therefore, the condition \eqref{eq1122wed} can be interpreted as $d \ge \dv \vec b$ in $\Omega$ and $\vec b \cdot \nu \ge 0$ on $\partial \Omega$ and \eqref{eq1123wed} translates to $d \ge \dv \vec c$ and $\vec c \cdot \nu \ge 0$.
We highlight that the inequality $d\geq\dv\vec b$ (in the sense of distributions) is not enough to guarantee the results in this article, and counterexamples are constructed in the Appendix. 

Heuristically, the condition \eqref{eq1122wed} is related to boundedness for the Neumann problem by the following reasoning.
For the Dirichlet problem, the condition $d\geq \dv \vec b$ translates to the fact that the constant function $u = 1$ is a supersolution, that is, $L1 \geq 0$.
In the case of the Neumann problem, taking into account the conormal boundary condition, the constant function $u=1$ is a supersolution when $d\geq\dv \vec b$ in $\Omega$ and also $\vec b \cdot \nu\geq 0$ on $\partial\Omega$.

We remark that, to the best of our knowledge, this is the first instance that the scale invariant $L^2$ theory is developed, and scale invariant pointwise estimates are shown in an optimal setting for the inhomogeneous terms and the Neumann data.
Unlike the Green function for the Dirichlet problem, for which there exists a rich literature, results on the Neumann Green function are not as common; 
see \cite{CK13a, CK13b, DiB10, LU68} and references therein.
The Neumann problem for equations with lower order terms is treated in \cite{DV09}, but the estimates there are not as optimal as ours.

%------------------------------------------------------------------------------%
\section{Preliminaries}
%------------------------------------------------------------------------------%
\subsection{Definitions}
%------------------------------------------------------------------------------%
Throughout the article, we will assume that $n \ge 3$.
If $\Omega\subset \bR^n$ is a domain, we denote by $W_1^p(\Omega)$ the Sobolev space of functions $u\in L^p(\Omega)$ such that their weak derivatives belong to $L^p(\Omega)$, with norm
\[
\norm{u}_{W_1^p(\Omega)}=\norm{u}_{L^p(\Omega)}+\norm{\nabla u}_{L^p(\Omega)}.
\]
With this norm, $W_1^2(\Omega)$ becomes a Hilbert space.
We will also use the space $Y_1^p(\Omega)$ for $1 \le p<n$, which is the completion of $C_c^{\infty}(\bR^n)$ under the norm
\[
\norm{u}_{Y_1^p(\Omega)}=\norm{u}_{L^{\frac{pn}{n-p}}(\Omega)}+\norm{\nabla u}_{L^p(\Omega)}.
\]
In the case when $\Omega$ is a Lipschitz domain (see Definition~\ref{LipDom}) with $\abs{\Omega}<+\infty$, the Sobolev inequality implies that $W_1^p(\Omega)=Y_1^p(\Omega)$ for $1 \le p <n$ as sets.

The Lorentz space $L^{p,q}(\Omega)$ consists of all measurable functions $f$ in $\Omega$ with $\norm{f}_{L^{p,q}(\Omega)}<\infty$, where 
\[
\norm{f}_{L^{p,q}(\Omega)}=\left\{\begin{array}{c l}
\displaystyle\left(\int_0^\infty \left(t^{\frac{1}{p}} f^*(t) \right)^q \frac{dt}{t} \right)^{\frac{1}{q}} &\text{ if }\;q<\infty,\\
\displaystyle\sup_{t>0} t^{\frac{1}{p}} f^*(t) & \text{ if }\;q=\infty,
\end{array}\right.
\]
and $f^*$ is the decreasing rearrangement of $f$. See \cite[\S 1.4.2]{Grafakos} for more details.
We denote by $W_1^{p,q}(\Omega)$ the set of functions $u\in L^{p,q}(\Omega)$ with $\nabla u\in L^{p,q}(\Omega)$.
We also define $Y_1^{p,q}(\Omega)$, for $1\le p <n$, as the set of functions $u \in L^{\frac{pn}{n-p},q}$ with $\nabla u \in L^{p,q}(\Omega)$, with the corresponding norm.

We say that an open, bounded and connected set $\Omega\subset \bR^n$ is a Lipschitz domain, if the part of $\Omega$ close to the boundary $\partial\Omega$ can be expressed as the part above graphs of Lipschitz functions.
In order to quantify our results, we use the following definition from \cite{KS11} (see also \cite[p. 189]{Ste70}).

\begin{definition}	\label{LipDom}
Let $\Omega\subset \bR^n$ be open and connected.
We say that $\Omega$ is a Lipschitz domain with character $(M,N)$ for $M\ge 0$ and $N \in \mathbb N$, if there exists $r=r_0>0$ and $q_i \in\partial\Omega$ for $i\in \set{1, \ldots, N}$, such that
%\[
%\partial\Omega\subset \bigcup_{i=1}^{\infty}B_{r}(q_i)\quad\text{and}\quad \sum_{i=1}^{\infty}\chi_{B_{r}(q_i)}\le N.
%\]
\[
\partial\Omega\subset \bigcup_{i=1}^NB_{r}(q_i).
\]
Moreover, for each $i \in \set{1,\ldots,N}$ there exists a Lipschitz function $\psi_i:\bR^{n-1} \to \bR$ with $\psi_i(0)=0$ and $\norm{D\psi_i}_{\infty}\leq M$ such that, after rotation and translation, $q_i=0$ and
\[
B_{10(M+1)r}(q_i)\cap\Omega=B_{10(M+1)r}(q_i)\cap \set{(x',x_n)\in\bR^n: x'\in\bR^{n-1},\;x_n>\psi_i(x')}.
\]
\end{definition}

It is straightforward to see that dilations do not change the character of a Lipschitz domain, thus making them a good candidate for a scale invariant theory.

\begin{remark}			\label{r_0Bounds}	
If $\Omega$ is a bounded Lipschitz domain with $\abs{\Omega}=1$, then $r_0$ in the above definition satisfies $c_1\leq r_0\leq c_2$ and $c_1\leq \diam\Omega\leq c_2$ for some positive numbers $c_1$ and $c_2$ depending only on $n$ and the Lipschitz character of $\Omega$.
See \cite[Lemmas 2.2 and 2.5]{Sak19} for the details.
\end{remark}

We now define solutions to the Neumann problem.
The conormal derivative of $L$ on the boundary $\partial \Omega$ is given by
\begin{equation*}		%		\label{eq1344fri}
(\mathbf A \nabla u + \vec b u)\cdot \nu,
\end{equation*}
where $\nu$ denotes the unit outer normal vector on $\partial\Omega$.
We will allow our functions to satisfy the Neumann condition only on some part of the boundary.
Suppose that $\Omega$ is a Lipschitz domain and $\Gamma \subset \partial\Omega$.
Let $f\in L^{\frac{2n}{n+2}}(\Omega)$, $\vec F=(F^1,\ldots, F^n) \in L^2(\Omega)$ and $g\in L^{2-\frac{2}{n}}(\Gamma)$.
We shall say that $u\in W_1^2(\Omega)$ is a solution to the Neumann problem
\begin{equation}		\label{eq2312fri}
\left\{\begin{array}{c l}
-\dv(\mathbf A \nabla u+\vec bu)+\vec c \cdot \nabla u+du = f-\dv \vec F &\text{ in }\;\Omega,\\
(\mathbf A \nabla u + \vec b u)\cdot \nu = g+ \vec F \cdot \nu & \text{ on }\;\Gamma,
\end{array}\right.
\end{equation}
if for any $\phi\in C_c^{\infty}(\bR^n \setminus (\partial\Omega \setminus \Gamma))$, we have
\begin{equation}			\label{eq2133sat}
\int_\Omega \mathbf A\nabla u \cdot \nabla \phi+\vec b u \cdot \nabla\phi+\vec c \phi \cdot \nabla u+du\phi = \int_\Omega f\phi+\vec F \cdot \nabla\phi +\int_\Gamma g\phi.
\end{equation}
We say that $u\in W_1^2(\Omega)$ is a subsolution of the Neumann problem \eqref{eq2312fri} if for any nonnegative function $\phi\in C_c^{\infty}(\bR^n)$ with $\phi\equiv 0$
we have
\begin{equation}			\label{eq2134sat}
\int_\Omega \mathbf A\nabla u \cdot \nabla \phi+\vec b u\cdot \nabla\phi+\vec c \phi \cdot \nabla u+du\phi \le \int_\Omega f\phi+\vec F \cdot \nabla\phi +\int_\Gamma g\phi.
\end{equation}
We say that $u\in W_1^2(\Omega)$ is a supersolution if the the inequality above is reversed.

\subsection{Special Lipschitz domains and the reflection method}
To show estimates close to the boundary, we use the geometry of Lipschitz domains, and extend the solutions by reflecting.
For this, we consider special Lipschitz domains.
For $r>0$, let  $B_r' =B_{r}'(0) \subset \bR^{n-1}$ be the $n-1$ dimensional ball of radius $r$ centered at $0$.
Let $\psi: B_r' \to \bR$ be a Lipschitz function with $\psi(0)=0$ and $\norm{\nabla \psi}_\infty \le M$.
We define
\begin{equation}	\label{eq:SpecialLipschitz}
\Omega^{+}_r=\Omega^{+}_r(0;\psi):=\set{(x',x_n)\in \bR^n: \abs{x'}<r,\;\psi(x')<x_n<\psi(x')+(M+1)r}.
\end{equation}
It is straightforward to check that special Lipschitz domains are Lipschitz domains (according to Definition~\ref{LipDom}), with Lipschitz character $(M,N)$, such that $N$ only depends on $n$ and $M$.
Then, consider the function $\vec \Psi:\Omega^{+}_r \to\bR^n$, defined by
\begin{equation}\label{eq:Psi}
\vec \Psi(x',x_n)=(x',2\psi(x')-x_n),
\end{equation}
which maps $\Omega^{+}_r$ onto its reflection
\[
\Omega^{-}_r=\Omega^{-}_r(0; \psi):=\set{(x',x_n)\in \bR^n: \abs{x'}<r,\;\psi(x')-(M+1)r<x_n<\psi(x')}.
\]
Note that $\vec\Psi$ is invertible, with $\abs{\det D\vec \Psi}=1$.
Let
\[
\partial_b \Omega^{+}_r=\partial_b \Omega^{+}_r(0; \psi):=\set{(x, \psi(x')): x'\in B_r'}
\]
denote the bottom part of the boundary of $\Omega^{+}_r$ and define the domain
\begin{equation}	\label{eq:OmegaR}
\begin{aligned}
\Omega_r&=\Omega_r(0;\psi)=\Omega^{+}_r(0; \psi) \cup \partial_b \Omega^{+}_r(0; \psi) \cup \Omega^{-}_r(0; \psi)\\
&=\set{(x',x_n)\in\bR^n: \abs{x'}<r,\; \abs{x_n-\psi(x')} < (M+1)r}.
\end{aligned}
\end{equation}
Note that $\vec \Psi$ given by the formula \eqref{eq:Psi} actually maps $\Omega_r$ onto $\Omega_r$ and $\vec \Psi^{-1}=\vec \Psi$.
	
Now, suppose for a Lipschitz function $\phi$ on $\Omega_r^+$, we have
\begin{equation}\label{eq:1st}
\int_{\Omega^{+}_r}\mathbf A\nabla u \cdot \nabla\phi+\vec b u \cdot \nabla\phi+\vec c \phi \cdot \nabla u+du\phi\leq\int_{\Omega^{+}_r}f\phi+\vec F\cdot \nabla\phi
\end{equation}
Then, by a change of variables, we have (note that $\abs{\det D\vec\Psi^{-1}}=1$)
\begin{equation}\label{eq:2nd}
\int_{\Omega^{-}_r} \mathbf{A}'\nabla u'\cdot \nabla\phi'+\vec{b}' u'\cdot \nabla\phi'+\vec{c}'\phi' \cdot \nabla u'+ d' u' \phi' \leq\int_{\Omega^{-}_r} f'\phi'+\vec{F}'\cdot \nabla\phi',
\end{equation}
where we set
\begin{equation}\label{eq:Reflections}
\begin{gathered}
u'(y)=u(\vec\Psi^{-1}(y)),\quad \phi'(y)=\phi(\vec\Psi^{-1}(y)),\\
\mathbf{A}'(y)=D\vec\Psi(\vec\Psi^{-1}(y))\,\mathbf A(\vec \Psi^{-1}(y))\,D\vec\Psi(\vec\Psi^{-1}(y))\tran,\\
\vec{b}'(y)= D\vec\Psi(\vec\Psi^{-1}(y))\, \vec b(\vec\Psi^{-1}(y)),\quad \vec{c}'(y)= D\vec\Psi(\vec\Psi^{-1}(y)) \,\vec c(\vec\Psi^{-1}(y)),\\
d'(y)=d(\vec\Psi^{-1}(y)),\quad
f'(y)=f(\vec\Psi^{-1}(y)),\quad \vec{F}'(y)= D\vec\Psi(\vec\Psi^{-1}(y))\,\vec F(\vec\Psi^{-1}(y)).
\end{gathered}
\end{equation}
We note that $\mathbf{A}'$ satisfies \eqref{ellipticity} with $\lambda$ and $\Lambda$ replaced by $c \lambda$ and $\Lambda/c$, where $c>0$ is a positive constant depending only on $n$ and $M$.
Also, we have
\[
\norm{\vec{b}'}_{L^n(\Omega^{-}_r)}\leq C \norm{\vec b}_{L^n(\Omega^{+}_r)},\quad \norm{\vec{c}'}_{L^n(\Omega^{-}_r)}\leq C\norm{\vec c}_{L^n(\Omega^{+}_r)},\quad 
\norm{d'}_{L^{n/2}(\Omega^{-}_r)}\leq C \norm{d}_{L^{n/2}(\Omega^{+}_r)},
\]
where $C$ only depends on $n$ and $M$.
Hence, we are led to the following lemma.
	
\begin{lemma}\label{Extension}
Let $\Omega_r^{+}$ be a special Lipschitz domain as in \eqref{eq:SpecialLipschitz}.
Assume $\mathbf A=(a^{ij})$ satisfy the uniform ellipticity and boundedness condition \eqref{ellipticity}, $\vec b$, $\vec c \in L^n(\Omega_r^+)$, $d \in L^{n/2}(\Omega_r^+)$.
Suppose $u\in W_1^2(\Omega^{+}_r)$ is a subsolution of
\[
\left\{\begin{array}{c l}
-\dv(\mathbf A \nabla u+\vec bu)+\vec c \cdot \nabla u+du = f-\dv \vec F &\text{ in }\;\Omega_r^+,\\
(\mathbf A \nabla u + \vec b u)\cdot \nu =  \vec F \cdot \nu & \text{ on }\;\partial_b\Omega_r^+,
\end{array}\right.
\]
where $f\in L^{\frac{2n}{n+2}}(\Omega_r^+)$ and $\vec F\in L^2(\Omega_r^+)$.
Let $u'=u\circ\vec\Psi^{-1}$, where $\vec\Psi$ is defined in \eqref{eq:Psi}, and define $\mathbf{A}'$, $\vec{b}'$, $\vec{c}'$, $d'$, $f'$, $\vec{F}'$ as in \eqref{eq:Reflections}.
Let us set $\tilde{u}=u$ in $\Omega_r^+$, $\tilde{u}=u'$ in $\Omega_r^-$, and similarly for $\tilde{\mathbf A}$, $\tilde{\vec b}$, $\tilde{\vec c}$, $\tilde{d}$, $\tilde{f}$, and $\tilde{\vec F}$.
Then, we have $\tilde{u}\in W_1^2(\Omega_r)$, where $\Omega_r$ is as in \eqref{eq:OmegaR}, and it is a subsolution of
\[
-\dv(\tilde{\mathbf A}\nabla\tilde{u}+\tilde{\vec b}\tilde{u})+\tilde{\vec c}\cdot\nabla\tilde{u}+\tilde{d}\tilde{u}=\tilde{f}-\dv\tilde{\vec F}\quad\text{in }\;\Omega_r.
\]
\end{lemma}
\begin{proof}
The fact that $\tilde{u}\in W_1^2(\Omega_r)$ is shown in \cite[p. 135]{EG15}.
Let $\phi$ be any Lipschitz function satisfying $\phi \ge 0$ and $\supp \phi \in \Omega_r$.
Consider the integral
\[
\int_{\Omega_r}\tilde{\mathbf A}\nabla \tilde{u} \cdot \nabla\phi+\tilde{\vec b} \tilde{u} \cdot \nabla\phi+\tilde{\vec c} \phi \cdot \nabla\tilde{u}+\tilde{d}\tilde{u}\phi
\]
and write it as a sum of two integrals on $\Omega_r^+$ and $\Omega_r^-$.
The proof is complete from the calculations showing the equivalence of \eqref{eq:1st} and \eqref{eq:2nd}.
\end{proof}
	
\begin{lemma}\label{divForExtension}
Let $\Omega_r^+$ be a special Lipschitz domain.
For $\vec b$ and $d\in L^1(\Omega_r^+)$, consider the extensions $\tilde{\vec b}$ and $\tilde{d}$ as in Lemma~\ref{Extension}.
If $(\vec b,d)$ satisfies
\[
\int_{\Omega_r^+} \vec b\cdot\nabla\phi+d\phi\geq 0,\quad\forall  \phi\in C_c^{\infty}(\Omega_r),
\]
then $\tilde{d}\geq\dv\tilde{\vec b}$ in $\Omega_r$.
\end{lemma}
\begin{proof}
Take any Lipschitz function $\phi$ satisfying $\phi \ge 0$ and $\supp \phi \in \Omega_r$.
Then, an approximation argument shows that
\[
\int_{\Omega_r^+} \vec b\cdot\nabla\phi+d\phi\geq 0.
\]
Let $\vec \Psi$ be as in \eqref{eq:Psi} and consider the reflections $\vec{b}'$ and $d'$ as in \eqref{eq:Reflections}.
Then, as in \eqref{eq:2nd}), by a change of variables, we have
\[
\int_{\Omega^{-}_r}\vec{b}'\cdot\nabla\phi+d'\phi=\int_{\Omega_r^+}\vec{b}\cdot \nabla \phi'+d \phi',
\]
where we set $\phi'=\phi\circ \vec \Psi$.
By the previous observation, we see that the above integral is nonnegative. 
Therefore, we have
\[
\int_{\Omega_r} \tilde{\vec b}\cdot\nabla\phi+\tilde{d}\phi=\int_{\Omega_r^+} \vec b \cdot\nabla\phi+d \phi+\int_{\Omega_r^-} \vec{b}'\cdot\nabla\phi+d'\phi \ge 0.\qedhere
\]
\end{proof}

\begin{remark}
Note that the condition $d\geq\dv \vec b$ in $\Omega_r^+$ alone is not enough to ensure that $\tilde{d}\geq\dv\tilde{\vec b}$ in $\Omega_r$, and the stronger condition~\eqref{eq1122wed} has to be imposed.
This is clear by considering the case when $\vec b=\vec e_n$ and $d=0$ in the upper half space.
\end{remark}
	
We also have the following estimate close to boundary.
\begin{lemma}\label{ReflectionNorm}
Let $\Omega_R^{+}=\Omega_R^{+}(0;\psi)$ be a special Lipschitz domain.
For $u\in L^1(\Omega^{+}_R)$ set $\tilde{u}=u$ in $\Omega_R^+$ and $\tilde{u}=u'$ in $\Omega_R^-$ with $u'$ as defined in \eqref{eq:Reflections}.
For $r<\frac{R}{3(M+1)}$, we have 
\[
\fint_{B_r} \abs{\tilde{u}} \leq C\fint_{B_{3(M+1)r}\cap\Omega_R^+} \abs{u},
\]
where $B_r=B_r(0)$ and $C$ is a constant depending only on $M$.
\end{lemma}
\begin{proof}
We write
\begin{equation*}				%\label{eq1241sat}
\fint_{B_r} \abs{\tilde{u}}=\frac{C}{\abs{B_r}} \int_{B_r\cap\Omega^{+}_R} \abs{u}+ \frac{C}{\abs{B_r}} \int_{B_r\cap\Omega^{-}_R} \abs{u'}.
\end{equation*}
Let $\vec \Psi$ be as in \eqref{eq:Psi}.
Recall that $\vec \Psi(0)=0$, $\abs{\det D\vec\Psi}=1$, and $\vec \Psi^{-1}=\vec \Psi$.
By the change of variable $y=\vec \Psi(x)$, we have
\[
\int_{B_r\cap\Omega^{-}_R}|u'(y)|\,dy=\int_{\vec\Psi (B_r\cap\Omega^{-}_R)}|u(x)|\,dx.
\]
Note that if $x=(x',x_n)\in B_r\cap\Omega^{-}_R$, then we have
\begin{align*}
\abs{\vec\Psi(x)} &\leq \abs{\vec \Psi(x)-x}+\abs{x}=\abs{2\psi(x')-2x_n}+\abs{x}\\
&\leq 2\abs{\psi(x')-\psi(0)}+2\abs{x_n}+\abs{x} \leq 2M\abs{x'}+3\abs{x}.
\end{align*}
Hence, we have $\vec\Psi(B_r\cap\Omega^{-}_R)\subset B_{3(M+1)r}\cap\Omega_R^+$, and the estimate follows.
\end{proof}	

\subsection{Sobolev-Poincar\'e inequalities}
The following  lemma is well known.
\begin{lemma}\label{Poincare}
Let $\Omega$ be a bounded Lipschitz domain.
Then there exists $C_0>0$, depending only on $n$ and the Lipschitz character of $\Omega$, such that for every $u\in W_1^2(\Omega)$, we have
\[
\norm{u-\bar u}_{L^{\frac{2n}{n-2}}(\Omega)}\le  C_0 \norm{\nabla u}_{L^2(\Omega)},\quad\text{where }\; \bar u=\fint_\Omega u.
\]
\end{lemma}

The following lemma gives the exact dependence of the constants in \cite[Lemma 8.1]{DV09} on the given data, which we prove here for the completeness.

\begin{lemma}			\label{AmpleZero}
Let $\Omega$ be a bounded Lipschitz domain, and $u\in W_1^2(\Omega)$. Suppose $u= 0$ in $E\subset\Omega$, with $\abs{E}\ge \delta \abs{\Omega}$ for some $\delta>0$.
Then,
\[
\norm{u}_{L^{\frac{2n}{n-2}}(\Omega)}\le  C_0 \norm{\nabla u}_{L^2(\Omega)},
\]
where $C_0$ depends only on $n$, $\delta$, and the Lipschitz character of $\Omega$.
\end{lemma}
\begin{proof}
If $\bar u$ is the average of $u$ in $\Omega$, we estimate
\[
\abs{\bar u} \leq \frac{1}{\abs{\Omega}}\int_\Omega \abs{u}=\frac{1}{\abs{\Omega}} \int_{\Omega \setminus E} \abs{u} \leq \frac{\abs{\Omega \setminus E}^{\frac{n+2}{2n}}}{\abs{\Omega}}\left(\int_\Omega \abs{u}^{\frac{2n}
{n-2}}\right)^{\frac{n-2}{2n}}=\frac{\abs{\Omega \setminus E}^{\frac{n+2}{2n}}}{\abs{\Omega}}\, \norm{u}_{\frac{2n}{n-2}}.
\]
Hence, combining with Lemma~\ref{Poincare}, we obtain
\begin{align*}
\norm{u}_{\frac{2n}{n-2}}&\leq \norm{u-\bar u}_{\frac{2n}{n-2}}+\norm{\bar u}_{\frac{2n}{n-2}} \leq C \norm{\nabla u}_2+\abs{\Omega}^{\frac{n-2}{2n}} \abs{\bar u} \leq C\|\nabla u\|_2+\left(\frac{\abs{\Omega \setminus E}}{\abs{\Omega}}\right)^{\frac{n+2}{2n}} \norm{u}_{\frac{2n}{n-2}}\\
&\leq C \norm{\nabla u}_2+(1-\delta)^{\frac{n+2}{2n}} \norm{u}_{\frac{2n}{n-2}},
\end{align*}
and the proof follows from rearranging the terms.
\end{proof}

\subsection{Trace inequalities}
The next lemma is well known.
We present the proof here for completeness.

\begin{lemma}			\label{trace_lemma}
Let $\Omega$ be a bounded Lipschitz domain.
If $u\in W_1^p(\Omega)$ with $1\le p<n$, then the trace of $u$ on $\partial\Omega$ belongs to $L^{\frac{p(n-1)}{n-p}}(\partial\Omega)$ and we have
\[
\norm{u}_{L^{\frac{p(n-1)}{n-p}}(\partial\Omega)}\leq C \norm{u}_{Y_1^p(\Omega)},
\]
where $C$ depends only on $n$, $p$, and the Lipschitz character of $\Omega$.
\end{lemma}
\begin{proof}
We may assume that $\abs{\Omega}=1$ since the inequality we want to establish  is scale invariant.
Also, we will first assume that $u$ is Lipschitz continuous.
Let $\psi_i$ be as in Definition~\ref{LipDom}.
For $x'\in B'_{2r_0}:=B_{2r_0}\cap \set{x_n=0}$ and $x_n\in(0,r_0)$, we have
\begin{align*}
\abs{u(x',\psi_i(x'))}&=\Abs{u(x',\psi_i(x')+x_n)-\int_{\psi_i(x')}^{\psi_i(x')+x_n}D_n u(x',t)\,dt}\\
&\leq \abs{u(x',\psi_i(x')+x_n)}+\int_{\psi_i(x')}^{\psi_i(x')+r_0}\abs{D_n u(x',t)}\,dt.
\end{align*}
Integrating the above with respect to $x_n$ on $(0, r_0)$ and dividing by $r_0$, we obtain
\[
\abs{u(x',\psi_i(x'))} \leq \frac{1}{r_0}\int_{\psi_i(x')}^{\psi_i(x')+r_0} \abs{u(x',x_n)}\,dx_n+\int_{\psi_i(x')}^{\psi_i(x')+r_0} \abs{D_n u(x',t)}\,dt,
\]
and, then integrating with respect to $x'$ on $B'_{2r_0}$, we have
\begin{multline*}
\int_{B'_{2r_0}} \abs{u(x',\psi_i(x'))}\,dx'
\leq\frac{1}{r_0}\int_{B'_{2r_0}}\int_{\psi_i(x')}^{\psi_i(x')+r_0} \abs{u(x',x_n)}\,dx_n dx' \\
+\int_{B'_{2r_0}}\int_{\psi_i(x')}^{\psi_i(x')+r_0} \abs{D_n u(x',t)}\,dt dx'
\le \frac{1}{r_0}\int_{\Omega} \abs{u}+\int_{\Omega} \abs{\nabla u}.
\end{multline*}
Therefore, we have
\[
\int_{\partial\Omega\cap B_{2r_0(q_i)}} \abs{u}\ \le  \frac{\sqrt{M^2+1}}{r_0} \int_{\Omega} \abs{u}+ \sqrt{M^2+1} \int_{\Omega} \abs{\nabla u}.
\]
Adding the above inequalities for $i=1,\ldots, N$, we obtain
\begin{equation}\label{eq:forSmooth}
\int_{\partial\Omega} \abs{u} \le C\int_{\Omega} \abs{u}+C\int_{\Omega} \abs{\nabla u},
\end{equation}
where $C$ depends only on $n$ and the Lipschitz character of $\Omega$; see Remark~\ref{r_0Bounds}.

Next, for $u\in W_1^p(\Omega)$, let $\set{\psi_m} \in C^{1}(\bR^n)$ be such that $\psi_m\to u$ in $W_1^p(\Omega)$.
By \eqref{eq:forSmooth} applied to $\abs{\psi_m}^{\frac{p(n-1)}{n-p}}$, we have
\[
\int_{\partial\Omega} \abs{\psi_m}^{\frac{p(n-1)}{n-p}} \le C\int_{\Omega} \abs{\psi_m}^{\frac{p(n-1)}{n-p}}+C\int_{\Omega}\tfrac{p(n-1)}{n-p}\, \abs{\psi_m}^{\frac{p(n-1)}{n-p}-1}\, \abs{\nabla\psi_m}.
\]
Then, by H\"older's inequality and Young's inequality, we obtain (recall  $\abs{\Omega}=1$)
\[
\norm{\psi_m}_{L^{\frac{p(n-1)}{n-p}}(\partial\Omega)}\leq C\norm{\psi_m}_{L^{\frac{np}{n-p}}(\Omega)}+C\norm{\nabla\psi_m}_{L^p(\Omega)},
\]
where $C$ depends on $n$, $p$ and the Lipschitz character of $\Omega$.
The proof is complete by letting $m\to\infty$.
\end{proof}

The next lemma is a variant of the previous lemma.
\begin{lemma}			\label{intpl_lemma}
Let $\Omega$ be a Lipschitz domain.
If $u\in W_1^{p,\infty}(\Omega)$ with  $1<p<n$, then the trace of $u$ on $\partial\Omega$ belongs to $L^{\frac{p(n-1)}{n-p},\infty}(\partial\Omega)$ and
\[
\norm{u}_{L^{\frac{p(n-1)}{n-p},\infty}(\partial\Omega)}\leq C\norm{u}_{Y_1^{p,\infty}(\Omega)},
\]
where $C$ depends only on $n$, $p$, and the Lipschitz character of $\Omega$.
\end{lemma}
\begin{proof}
Since the estimate is scale-invariant, so we may assume that $\abs{\Omega}=1$.
Then since $\norm{u}_{W_1^{p,\infty}(\Omega)}\leq \norm{u}_{Y_1^{p,\infty}(\Omega)}$, it is enough to show that
\[
\norm{u}_{L^{\frac{p(n-1)}{n-p},\infty}(\partial\Omega)}\leq C\norm{u}_{W_1^{p,\infty}(\Omega)}.
\]
By \cite[Theorem 2]{DS79} and \cite[Theorem 1.9, p. 300]{BS88}, we have
\[
W_1^{p,\infty}(\Omega)=\bigl(W_1^1(\Omega),W_1^\infty(\Omega)\bigr)_{1-1/p,\infty}\quad\text{and}\quad
W_1^{p,1}(\Omega)=\bigl(W^1_1(\Omega),W^\infty_1(\Omega)\bigr)_{1-1/p,1}.
\]
Also, by \cite[Corollary V.5.13]{BS88}), we have
\[
W_1^p(\Omega)=\bigl(W_1^1(\Omega),W^\infty_1(\Omega)\bigr)_{1-1/p, p}.
\]
Fix $\varepsilon>0$ so small that $p-\varepsilon>1$ and $p+\varepsilon<n$.
Then, choose $\theta\in(0,1)$ such that
\[
\frac{1-\theta}{p-\varepsilon}+\frac{\theta}{p+\varepsilon}=\frac{1}{p}.
\]
By the reiteration theorem (see \cite[Theorem 2.4, p. 311]{BS88}), we have
\[
\left(W_1^{p-\varepsilon}(\Omega),W_1^{p+\varepsilon}(\Omega)\right)_{\theta,\infty}=\left(W_1^1(\Omega),W_1^\infty(\Omega)\right)_{\theta',\infty}=W_1^{p,\infty}(\Omega),
\]
where
\[
\theta'=(1-\theta)\left(1-\frac{1}{p-\varepsilon}\right)+\theta\left(1-\frac{1}{p+\varepsilon}\right)=1-\frac{1-\theta}{p-\varepsilon}-\frac{\theta}{p+\varepsilon}=1-\frac{1}{p}.
\]
Similarly, for $1<q_0<q_1<\infty$, we have
\[
\bigl(L^{q_0}(\partial\Omega), L^{q_1}(\partial\Omega)\bigr)_{\theta,\infty}=\left(L^1(\partial\Omega),L^\infty(\partial\Omega)\right)_{\theta',\infty}=L^{q_\theta,\infty}(\partial\Omega),
\]
where
\[
\frac{1}{q_\theta}=\frac{1-\theta}{q_0}+\frac{\theta}{q_1}.
\]
Let $T$ be the trace operator. By Lemma~\ref{trace_lemma} we have
 \[
\norm{Tu}_{L^{\frac{p(n-1)}{n-p}}(\partial\Omega)}\le  C\norm{u}_{W_1^p(\Omega)}.
\]
Then, by \cite[Theorem V.1.12]{BS88}, the trace operator maps
\[
T: \bigl(W_1^{p-\varepsilon}(\Omega),W_1^{p+\varepsilon}(\Omega)\bigr)_{\theta,\infty} \to \bigl(L^{\frac{(p-\varepsilon)(n-1)}{n-(p-\varepsilon)}}(\partial\Omega),L^{\frac{(p+\varepsilon)(n-1)}{n-(p+\varepsilon)}}(\partial\Omega)\bigr)_{\theta,\infty}
\]
and thus,
\[
T:W_1^{p,\infty}(\Omega)\to L^{q,\infty}(\partial\Omega),
\]
where
\[
\frac{1}{q}=(1-\theta)\,\frac{n-p+\varepsilon}{(p-\varepsilon)(n-1)}+\theta\,\frac{n-p-\varepsilon}{(p+\varepsilon)(n-1)}=\frac{n-p}{(n-1)p}
\]
and $\norm{T}$ depends only on $n$, $p$, and the Lipschitz character of $\Omega$.
\end{proof}

\subsection{The splitting lemmas}
The estimates we show in this article depend on the lower order coefficients only via  norms.
This is straightforward to do if some smallness for the lower order coefficients is involved.
To pass to larger norms, we consider a splitting $\set{u_j}_{j=1}^N$ of $u$ such that $\nabla u_j$ are supported in sets where the lower order coefficients have small norms.
This is the context of the following lemma, which first appeared in \cite{BM76} and was later extended to a more general setting in \cite{Mou19}.

\begin{lemma}		\label{Split}
Let $\Omega\subset \bR^n$ be an open set, $h \in L^n(\Omega)$,  and $u\in Y_1^2(\Omega)$.
For any $\varepsilon>0$, there exist mutually disjoint measurable sets $\Omega_i\subset \Omega$ and functions $u_i \in Y_1^2(\Omega)$ for $i=1,\ldots, N$ with the following properties.
\begin{enumerate}
\item $\norm{h}_{L^n(\Omega_i)}=\varepsilon\,$ for $\,i=1,\dots, N-1\,$ and $\norm{h}_{L^n(\Omega_N)}\leq \varepsilon$,
\item $\set{x \in \Omega : \nabla u_i(x)\neq 0}\subset\Omega_i$,
\item $\nabla u=\nabla u_i$ in $\Omega_i$,
\item $\abs{u_i} \leq \abs{u}$,
\item $u u_i\geq 0$,
\item $u=\sum_{i=1}^Nu_i$,
\item $u_i\nabla u=\sum_{j=1}^iu_i\nabla u_j$,
\item $u\nabla u_i=\sum_{j=i}^N u_j\nabla u_i$,
\end{enumerate}
and $N$ has an upper bound $N \le 1+\left(\norm{h}_n/\varepsilon\right)^n$.
\end{lemma}
We will also need the following lemma, which splits a function $u$ with integral zero to functions $u_i$ as in Lemma~\ref{Split} above such that all $u_i$ have integral zero.

\begin{lemma}\label{SplitMeanZero}
Let $\Omega\subset \bR^n$ be an open set, $h \in L^n(\Omega)$,  and $u\in Y_1^2(\Omega) \cap L^1(\Omega)$ satisfying $\int_\Omega u=0$.
For any $\varepsilon>0$, there exist mutually disjoint measurable sets $\Omega_i\subset \Omega$ and functions $u_i \in Y_1^2(\Omega)\cap L^1(\Omega)$ for $i=1,\ldots, N$ satisfying all the properties in Lemma~\ref{Split} and additionally $\int_\Omega u_i=0$.
\end{lemma}
\begin{proof}
Set $m=\essinf_\Omega u$ and $M=\esssup_\Omega u$.
Since $\int_\Omega u=0$, unless $u\equiv 0$, we have $m < 0$ and $M > 0$.
Consider the function $g: (m,0] \to \bR$ defined by
\[
g(k)=\int_{\set{u<k}}(k-u).
\]
We note that $g \ge 0$, $g(0)=\int_\Omega u^{-}$, and $\lim_{k \to m} g(k)=0$.
Moreover, $g$ is strictly increasing and continuous on $(m,0]$.
Indeed, for $m<k_1<k_2\leq 0$, we have
\begin{align*}
g(k_2)-g(k_1)&=\int_{\set{u<k_2}}(k_2-u)-\int_{\set{u<k_1}}(k_1-u)\\
&=\int_{\set{u\leq k_1}}(k_2-u)+\int_{\set{k_1<u<k_2}}(k_2-u)-\int_{\set{u\leq k_1}}(k_1-u)\\
&=(k_2-k_1)\,\Abs{\set{u\leq k_1}}+\int_{\set{k_1<u<k_2}}(k_2-u).
\end{align*}
We shall set $g(m)=0$.
Next, we claim that for $0 \le  l<M$, there exists a unique $k_l\in(m,0]$ such that
\begin{equation}\label{eq:sl}
\int_{\set{u>l}}(u-l)=\int_{\set{u<k_l}}(k_l-u)=g(k_l).
\end{equation}
Indeed, since
\[
g(m)=0<\int_{\set{u>l}}(u-l) \le \int_{\set{u>l}}u \le \int_{\set{u\geq 0}}u=\int_\Omega u^{+}=\int_\Omega u^{-}=g(0),
\]
there exists a unique $k_l\in(m,0]$ satisfying \eqref{eq:sl}.
We set $k_M=m$ so that $g(k_M)=0$.
Notice that $0 \ge k_{l_1} > k_{l_2} \ge m$ if $0\le l_1 < l_2 \le M$.
For $0\leq s<t \le M$, we set
\[
\Omega(s,t)=\set{ s<u \leq t, \;\nabla u \neq 0} \cup \set{k_t\leq u<k_s,\; \nabla u \neq 0},
\]
and define
\[
h(s,t)=\int_{\Omega(s,t)} \abs{h}^n.
\]
Also, for $0 \le s<M$, we define
\[
u_{s,M}=\left\{\begin{array}{c l} u-k_s, & u\leq k_s \\ 0, & k_s<u\leq s \\u-s, & u>s \end{array}\right.
\]
and for $0 \le s <t < M$, we set
\[
u_{s,t}=u_{s,M}-u_{t,M}.
\]
Observe that \eqref{eq:sl} implies
\[
\int_\Omega u_{s,M}=\int_{\set{u>s}}(u-s)+\int_{\set{u<k_s}}(u-k_s)=0.
\]
Since $u_{s,t}=u_{s,M}-u_{t,M}$, we see that $\int_\Omega u_{s,t}=0$ for any $s$, $t$ satisfying $0 \le s <t  \le M$.
Note that we actually have
\[
u_{s,t}=\left\{\begin{array}{c l} k_t-k_s, & u\leq k_t\\ u-k_s, & k_t<u\leq k_s \\ 0, & k_s<u\leq s \\u-s, & s<u\leq t \\t-s, & u>t\end{array}\right..
\]
As in \cite[Lemma 2.31]{Mou19}, we see that $s \mapsto h(s,t)$ is continuous for any fixed $t$.

Now, we set $s_0=M$ and let $N$ be the smallest integer satisfying $\norm{h}_{L^n(\Omega(0,s_0))}^n<N\varepsilon^n$.
If $N=1$, then we set $s_1=0$, $\Omega_1=\Omega(s_1,s_0)=\Omega(0,M)$, and stop.
If $N \ge 2$, then we have $h(0,s_0)=\norm{h}_{L^n(\Omega(0,s_0))}^n\geq (N-1) \varepsilon^n\ge \varepsilon^n$, and thus continuity of $h(\cdot,s_0)$ implies that there exists $s_1 \in [0, s_0)$ such that $h(s_1,s_0)=\varepsilon^n$.
Set $\Omega_1=\Omega(s_1,s_0)=\Omega(s_1,M)$.
If $N=2$, then we set $s_2=0$, $\Omega_2=\Omega(s_2,s_1)=\Omega(0,s_1)$, and stop.
If $N \ge 3$, then
\[
h(0,s_1)=\int_{\Omega(0,s_1)} \abs{h}^n=\int_{\Omega(0,s_0)} \abs{h}^n-\int_{\Omega(s_1,s_0)} \abs{h}^n\geq (N-2)\varepsilon^n \ge \varepsilon^n,
\]
and thus there exists $s_2\in [0,s_1)$ such that $h(s_1,s_2)=\varepsilon^n$.

Inductively, we construct a sequence $M=s_0>s_1> \cdots > s_{N-1}>s_N=0$ such that $h(s_i,s_{i-1})=\varepsilon^n$ for $i=1,\ldots, N-1$, and $h(s_N,s_{N-1})\le \varepsilon^n$. Set $\Omega_i=\Omega(s_i,s_{i-1})$ for $i=1,\dots N$, and $u_i=u_{s_i,s_{i-1}}$. Then $\nabla u_i$ is supported in $\Omega_i$, and the rest of the relations are straightforward to verify.
\end{proof}

\subsection{The main estimate}
The following lemma treats the main estimate that we will use in the proof of our main results. 
In the case of the Dirichlet problem, a similar result is contained in \cite{Mou19}.

\begin{lemma}\label{Main}
Let $\Omega\subset \bR^n$ be a bounded Lipschitz domain, with $\abs{\Omega}\leq 1$.
Let $\mathbf A=(a^{ij})$ satisfy the uniform ellipticity and boundedness condition \eqref{ellipticity}, $\vec b$, $\vec c \in L^n(\Omega)$, $d \in L^{n/2}(\Omega)$.
Assume that either $(\vec b, d)$ satisfies \eqref{eq1122wed} or $(\vec c,d)$ satisfies \eqref{eq1123wed}.
\begin{enumerate}
\item 
For $f\in L^{\frac{2n}{n+2}}(\Omega)$, $\vec F\in L^2(\Omega)$, and $g\in L^{2-\frac{2}{n}}(\partial\Omega)$, let $u\in Y_1^2(\Omega)$ be a subsolution to the problem
\begin{equation}				\label{eq1523tue}
\left\{\begin{array}{c l}
-\dv(\mathbf A \nabla u+\vec bu)+\vec c \cdot \nabla u+du = f-\dv \vec F &\text{ in }\;\Omega,\\
(\mathbf A \nabla u + \vec b u)\cdot \nu = g+ \vec F \cdot \nu & \text{ on }\;\partial\Omega,
\end{array}\right.
\end{equation}
Consider the splitting $\set{u_i}_{i=1}^N$ of $u^{+}$ corresponding to $h=\abs{\vec b-\vec c}\in L^n(\Omega)$ and $\varepsilon \in(0,\frac18 \lambda]$ as in Lemma~\ref{Split}.
Suppose, for some numbers $a_i\geq 0$ and some constant $C_0$ with $C_0 \varepsilon \le \frac{\lambda}{8}$, we have
\begin{equation}		\label{eq:relation}
\norm{u_i-a_i}_{L^{\frac{2n}{n-2}}(\Omega)}\leq C_0 \norm{\nabla u_i}_\Omega,\quad i=1,\ldots, N.
\end{equation}
Then, we have
\begin{equation}		\label{eq:MainEstimate}
\int_\Omega \Abs{\nabla u^{+}}^2\leq C\left( a^2+ \norm{f^{+}}_{L^{\frac{2n}{n+2}}(\Omega)}^2+\norm{g^{+}}_{L^{2-\frac{2}{n}}(\partial\Omega)}^2+\norm{\vec F}_{L^2(\Omega)}^2\right),
\end{equation}
where $a=\sum_{i=1}^N a_i$, and the constant $C$ depends on $n$, $\lambda$, $\norm{\vec b-\vec c}_n$, $C_0$, $\varepsilon$, and the Lipschitz character of $\Omega$.
\item
Let $u\in Y_1^2(\Omega)$ be a solution to the same Neumann problem  as above and consider the splitting $\set{u_i}_{i=1}^N$ of $u$ corresponding to $h=\abs{\vec b-\vec c}\in L^n(\Omega)$ and $\varepsilon \in(0,\frac18 \lambda)$.
Assume also that \eqref{eq:relation} holds for some numbers $a_i\geq 0$ and $C_0$ with $C_0\varepsilon <\frac{\lambda}{8}$.
Then, we have
\[
\int_\Omega \Abs{\nabla u}^2\leq C\left( a^2+ \norm{f}_{L^{\frac{2n}{n+2}}(\Omega)}^2+\norm{g}_{L^{2-\frac{2}{n}}(\partial\Omega)}^2+\norm{\vec F}_{L^2(\Omega)}^2\right),
\]
where $a=\sum_{i=1}^N a_i$ and $C$ depends on $n$, $\lambda$, $\norm{\vec b-\vec c}_n$, $C_0$, $\varepsilon$, and the Lipschitz character of $\Omega$.
\end{enumerate}
\end{lemma}
\begin{proof}
We first deal with the case when $(\vec b, d)$ satisfies \eqref{eq1122wed}.
Let $u$ be a subsolution to the problem \eqref{eq1523tue} and $\set{u_i}_{i=1}^N$ be the splitting of $u^{+}$.
By properties (d) and (e) in Lemma~\ref{Split}, we have $u_i \ge 0$ and $u u_i \ge 0$ since $u_i=0$ whenever $u^{+}=0$.
Therefore, using $uu_i\ge 0$ as a test function in \eqref{eq1122wed}, we get
\begin{equation}\label{eq:ineq}
\int_\Omega \vec bu \cdot \nabla u_i+duu_i=\int_\Omega \vec b \cdot \nabla(uu_i)+duu_i-\int_\Omega \vec b u_i \cdot \nabla u \ge -\int_\Omega \vec b u_i \cdot \nabla u.
\end{equation}
Also, by using $u_i \ge 0$ as a test function in \eqref{eq2134sat}, we get
\[
\int_\Omega \mathbf A\nabla u \cdot \nabla u_i+\vec b u\cdot \nabla u_i+\vec c u_i  \cdot \nabla u +du u_i \le \int_\Omega f^{+}u_i+\vec F \cdot \nabla u_i +\int_{\partial\Omega} g^{+} u_i.
\]
Hence, we have
\begin{equation}			\label{eq2000sun}
\int_\Omega \mathbf A \nabla u \cdot \nabla u_i-(\vec b-\vec c)u_i \cdot \nabla u \le \int_\Omega f^{+}u_i+\vec F \cdot \nabla u_i +\int_{\partial\Omega} g^{+} u_i.
\end{equation}
By Lemma~\ref{Split}, we have (recall that $u_i=0$ when $u^{+}=0$),
\[
\int_\Omega \mathbf A \nabla u \cdot \nabla u_i = \int_{\Omega_i} \mathbf A \nabla u^{+} \cdot \nabla u_i= \int_{\Omega_i} \mathbf A \nabla u_i \cdot \nabla u_i
\]
and 
\begin{equation}			\label{eq2001sun}
\int_\Omega (\vec b-\vec c)u_i \cdot \nabla u= \int_\Omega  (\vec b-\vec c) u_i \cdot \nabla u^{+} = \sum_{j=1}^i \int_{\Omega_j} (\vec b-\vec c)u_i \cdot \nabla u_j.
\end{equation}
Therefore, we obtain from \eqref{eq2000sun} that
\begin{multline}\label{eq:ToPlug3}
\lambda\int_{\Omega_i} \abs{\nabla u_i}^2 \le \sum_{j=1}^i \int_{\Omega_j} (\vec b-\vec c) u_i \cdot \nabla u_j + \int_\Omega f^{+}u_i+\int_{\partial\Omega} g^{+}u_i+\int_\Omega\vec F\cdot\nabla u_i\\
=: I_i+J_i+K_i+L_i.
\end{multline}
First, we estimate $I_i$.
Note that by Lemma~\ref{Split}, we have
\[
\norm{\vec b-\vec c}_{L^n(\Omega_j)} \le \varepsilon.
\]
This together with H\"older's inequality and \eqref{eq:relation}, we obtain
\begin{equation}			\label{eq1340mon}
\begin{aligned}
\abs{I_i} & \le \Abs{\sum_{j=1}^i \int_{\Omega_j}(u_i-a_i)(\vec b-\vec c)\cdot \nabla u_j}+\Abs{\sum_{j=1}^i \int_{\Omega_i} a_i (\vec b-\vec c)\cdot \nabla u_j}\\
&\le \sum_{j=1}^i \varepsilon \norm{\nabla u_j}_2\,\norm{u_i-a_i}_{\frac{2n}{n-2}}+\sum_{j=1}^i \varepsilon a_i \norm{\nabla u_j}_2\\
& \leq C_0 \varepsilon \norm{\nabla u_i}_2\sum_{j=1}^i \norm{\nabla u_j}_2+\varepsilon a_i \sum_{j=1}^i  \norm{\nabla u_j}_2.
\end{aligned}
\end{equation}
Therefore, we have
\begin{equation}			\label{eq1234mon}
\abs{I_1} \le C_0 \varepsilon \norm{\nabla u_1}_2^2 +\varepsilon a \norm{\nabla u_1}_2 \le \frac{\lambda}{4} \norm{\nabla u_1}_2^2+ \frac{\lambda}{32} a^2.
\end{equation}
and by using Cauchy's inequality, for $i=2,\ldots, N$, we have
\begin{equation}			\label{eq1235mon}
\begin{aligned}
\abs{I_i} &\leq \frac{\lambda}{8} \norm{\nabla u_i}_2^2+\frac{\lambda}{8} \norm{\nabla u_i}_2 \sum_{j=1}^{i-1} \norm{\nabla u_j}_2+\frac{\lambda}{8} a \norm{\nabla u_i}_2 + \frac{\lambda}{8} a \sum_{j=1}^{i-1} \norm{\nabla u_j}_2\\
&\leq \frac{\lambda}{4}\|\nabla u_i\|_2^2+\frac{\lambda}{8}\left(\sum_{j=1}^{i-1}\norm{\nabla u_j}_2\right)^2+\frac{\lambda}{8} a^2.
\end{aligned}
\end{equation}
For $J_i$, we estimate
\begin{align*}
J_i&=\int_\Omega f^{+}(u_i-a_i)+a_i\int_\Omega f^{+}\\
&\le C_0 \norm{ f^{+}}_{\frac{2n}{n+2}}\, \norm{\nabla u_i}_2+a_i \norm{f^{+}}_1 \le \frac{\lambda}{8} \norm{\nabla u_i}_2^2+C\norm{f^{+}}_{\frac{2n}{n+2}}^2+a^2,
\end{align*}
where we used $\abs{\Omega} \leq 1$ and $C$ depends on $\lambda$ and $C_0$.
To estimate $K_i$, first note that the trace inequality for the constant function $1$ shows that (see Lemma~\ref{trace_lemma})
\begin{equation*}	%	\label{eq:sigmaEst}
\abs{\partial\Omega}^{\frac{n-2}{2n-2}}=\norm{1}_{L^{\frac{2n-2}{n-2}}(\partial\Omega)} \leq C \norm{1}_{Y_1^2(\Omega)} = C\norm{1}_{L^{\frac{2n}{n-2}}(\Omega)} \le C,
\end{equation*}
where $C$ depends on $n$ and the Lipschitz character of $\Omega$.
Also, we have
\[
\norm{u_i-a_i}_{L^{\frac{2n-2}{n-2}}(\partial\Omega)} \le C \norm{u_i-a_i}_{Y_1^2(\Omega)} \le C(1+C_0) \norm{\nabla u_i}_{L^2(\Omega)},
\]
where we used \eqref{eq:relation}.
Then, we get
\begin{align*}
K_i&=\int_{\partial\Omega}g^{+}(u_i-a_i)+a_i\int_{\partial\Omega}g^{+}\\
&\leq \norm{g^{+}}_{L^{2-\frac{2}{n}}(\partial\Omega)}\,\norm{u_i-a_i}_{L^{\frac{2n-2}{n-2}}(\partial\Omega)}+a_i \norm{g^{+}}_{L^{2-\frac{2}{n}}(\partial\Omega)}\,\norm{1}_{L^{\frac{2n-2}{n-2}}(\partial\Omega)}\\
&\leq \frac{\lambda}{8} \norm{\nabla u_i}_2^2+C\norm{g^{+}}_{L^{2-\frac{2}{n}}(\partial\Omega)}^2+Ca^2,
\end{align*}
where $C$ depends on $n$, $\lambda$, $C_0$, and the Lipschitz character of $\Omega$.
Finally,
\[
L_i=\int_\Omega \vec F \cdot \nabla u_i \leq \norm{\vec F}_2 \,\norm{\nabla u_i}_2 \leq \frac{\lambda}{4} \norm{\nabla u_i}_2+\frac{1}{\lambda}\norm{\vec F}_2^2.
\]
Gathering the estimates for $I_i$, $J_i$, $K_i$ and $L_i$, plugging in \eqref{eq:ToPlug3} and setting
\[
x_i:=\norm{\nabla u_i}_2,
\]
we obtain
\[
x_1^2 \le \frac{3}{4}x_1^2+C\left(a^2+\norm{f^{+}}_{\frac{2n}{n+2}}^2+\norm{g^{+}}_{L^{2-\frac{2}{n}}(\partial\Omega)}^2+\norm{\vec F}_2^2\right)
\]
and
\[
x_i^2 \leq \frac{3}{4}x_i^2+\frac{1}{8}\bigl(\sum_{j=1}^{i-1} x_j\bigr)^2+C\left(a^2+\norm{f^{+}}_{\frac{2n}{n+2}}^2+\norm{g^{+}}_{L^{2-\frac{2}{n}}(\partial\Omega)}^2+\norm{\vec F}_2^2\right),
\]
for $i=2,\ldots, N$.
Hence, we have
\begin{equation}			\label{eq:x_i}
\begin{aligned}
x_1 &\leq C\left(a+\norm{f^{+}}_{\frac{2n}{n+2}}+\norm{g^{+}}_{L^{2-\frac{2}{n}}(\partial\Omega)}+\norm{\vec F}_2 \right),\\
x_i &\leq \sum_{j=1}^{i-1} x_j+C\left(a+\norm{f^{+}}_{\frac{2n}{n+2}}+\norm{g^{+}}_{L^{2-\frac{2}{n}}(\partial\Omega)}+\norm{\vec F}_2 \right),\quad i=2,\ldots, N,
\end{aligned}
\end{equation}
for some $C>0$ that depends on $n$, $\lambda$, $C_0$, and the Lipschitz character of $\Omega$.
Then, an induction argument shows that
\begin{equation}			\label{eq1303mon}
x_{i} \leq 2^{i-1} C\left(a+\norm{f^{+}}_{\frac{2n}{n+2}}+\norm{g^{+}}_{L^{2-\frac{2}{n}}(\partial\Omega)}+\norm{\vec F}_2\right),\quad i=1,\ldots, N,
\end{equation}
and thus, we get (recall that the supports of $\nabla u_i$ are mutually disjoint)
\begin{equation}			\label{eq2248sun}
\int_\Omega \abs{\nabla u^{+}}^2=\sum_{i=1}^N\int_\Omega \abs{\nabla u_i}^2\leq 4^{N} C\left(a+\norm{f^{+}}_{\frac{2n}{n+2}}+\norm{g^{+}}_{L^{2-\frac{2}{n}}(\partial\Omega)}+\norm{\vec F}_2\right)^2.
\end{equation}
The proof for part (a) is complete using the bound on $N$ from Lemma~\ref{Split}.
	
The proof for part (b) is carried out in the same way.
Note that we have $uu_i \ge 0$ by Lemma~\ref{Split} and thus \eqref{eq:ineq}.
Then, by using $u_i$ as a test function in \eqref{eq2133sat} and applying \eqref{eq:ineq}, we still obtain \eqref{eq2000sun} and \eqref{eq:ToPlug3} (with $f,g$ in the place of $f^+,g^+$).
The rest of proof for part (b) is exactly the same except that in \eqref{eq2248sun} we now have 
\[
\int_\Omega \abs{\nabla u}^2=\sum_{i=1}^N\int_\Omega \abs{\nabla u_i}^2.
\]

Next, we consider the case when $(\vec c, d)$ satisfies \eqref{eq1123wed}.
The proof is essentially the same as that of the other case and requires only a minor modification.
Here, we present the proof for part (a) only since the proof for part (b) is similar.
Let $u$ be a subsolution to the problem \eqref{eq1523tue} and $\set{u_i}_{i=1}^N$ be the splitting of $u^{+}$.
Similar to \eqref{eq:ineq}, we have 
\[
\int_\Omega \vec c u_i \cdot \nabla u+duu_i=\int_\Omega \vec c \cdot \nabla(uu_i)+duu_i-\int_\Omega \vec c u \cdot \nabla u_i\ge -\int_\Omega \vec c u \cdot \nabla u_i,
\]
and thus instead of \eqref{eq2000sun}, we have
\[
\int_\Omega \mathbf A \nabla u \cdot \nabla u_i+(\vec b-\vec c) u \cdot \nabla u_i \le \int_\Omega f^{+}u_i+\vec F \cdot \nabla u_i +\int_{\partial\Omega} g^{+} u_i.
\]
Similar to \eqref{eq2001sun}, by Lemma~\ref{Split} we have 
\[
\int_\Omega (\vec b-\vec c)u\cdot \nabla u_i = \sum_{j=i}^{N} \int_{\Omega_i} (\vec b-\vec c)u_j \cdot \nabla u_i,
\]
and thus instead of \eqref{eq:ToPlug3}, we have
\[
\lambda\int_{\Omega_i} \abs{\nabla u_i}^2 \le \sum_{j=i}^N \int_{\Omega_i} (\vec c-\vec b)u_j\cdot \nabla u_i + \int_\Omega f^{+}u_i+\int_{\partial\Omega}g^{+}u_i+\int_\Omega\vec F\cdot\nabla u_i=:\tilde I_i + J_i + K_i +L_i.
\]
Similar to \eqref{eq1340mon}, we have
\begin{align*}
\abs{\tilde I_i} & \le \Abs{\sum_{j=i}^N \int_{\Omega_i}(u_j-a_j)(\vec b-\vec c)\cdot \nabla u_i}+\Abs{\sum_{j=i}^N \int_{\Omega_i} a_j (\vec b-\vec c)\cdot \nabla u_i}\\
&\le \sum_{j=i}^N \varepsilon \norm{\nabla u_i}_2\,\norm{u_j-a_j}_{\frac{2n}{n-2}}+\sum_{j=i}^N \varepsilon a_j \norm{\nabla u_i}_2\\
& \leq C_0 \varepsilon \norm{\nabla u_i}_2\sum_{j=i}^N \norm{\nabla u_j}_2+\varepsilon a \norm{\nabla u_i}_2.
\end{align*}
Therefore, similar to \eqref{eq1234mon} and \eqref{eq1235mon}, we have
\begin{align*}
\abs{\tilde I_N} &\le  \frac{\lambda}{4} \norm{\nabla u_N}_2^2+ \frac{\lambda}{32} a^2,\\
\Abs{\tilde I_i} &\le \frac{\lambda}{4}\|\nabla u_i\|_2^2+\frac{\lambda}{16}\left(\sum_{j=i+1}^N \norm{\nabla u_j}_2\right)^2+\frac{\lambda}{16} a^2,\quad i=1, \ldots, N-1.
\end{align*}
The estimates for $J_i$, $K_i$, and $L_i$ remain unchanged.
Then by keeping the same notation $x_i=\norm{\nabla u_i}_2$, we obtain, instead of \eqref{eq:x_i}, that
\[
\begin{aligned}
x_N &\leq C\left(a+\norm{f^{+}}_{\frac{2n}{n+2}}+\norm{g^{+}}_{L^{2-\frac{2}{n}}(\partial\Omega)}+\norm{\vec F}_2 \right),\\
x_i &\leq \sum_{j=i+1}^N x_j+C\left(a+\norm{f^{+}}_{\frac{2n}{n+2}}+\norm{g^{+}}_{L^{2-\frac{2}{n}}(\partial\Omega)}+\norm{\vec F}_2 \right),\quad i=1,\ldots, N-1.
\end{aligned}
\]
Then, similar to \eqref{eq1303mon}, we obtain
\[
x_{N-i} \leq 2^i C\left(a+\norm{f^{+}}_{\frac{2n}{n+2}}+\norm{g^{+}}_{L^{2-\frac{2}{n}}(\partial\Omega)}+\norm{\vec F}_2\right),\quad i=0,\ldots, N-1,
\]
rom which the desired conclusion follows again.
\end{proof}

\section{Existence and uniqueness of solutions}
First, let us consider the homogeneous Neumann problem
\begin{equation}				\label{eq0929tue}
\left\{\begin{array}{c l}-\dv(\mathbf A\nabla u+\vec b u)+\vec c\cdot \nabla u+du= 0 & \text{in }\;\Omega,\\
(\mathbf A \nabla u+\vec b u)\cdot \nu=0 & \text{on }\;\partial\Omega.\end{array}\right.
\end{equation}
In what follows, we will show that the solution space for the above problem has at most $1$ dimension.

\begin{lemma}			\label{subsol}
Let $\Omega\subset \bR^n$ be a bounded Lipschitz domain.
Let $\mathbf A=(a^{ij})$ satisfy the uniform ellipticity and boundedness condition \eqref{ellipticity}, $\vec b$, $\vec c \in L^n(\Omega)$, $d \in L^{n/2}(\Omega)$.
Assume that the pair $(\vec c, d)$ satisfies the condition \eqref{eq1123wed}.
If $u$ is a subsolution to the problem \eqref{eq0929tue}, then either $u\leq 0$ a.e. in $\Omega$, or $u>0$ a.e. in $\Omega$.
\end{lemma}
\begin{proof}
Suppose that $u>0$ in a set of positive measure.
We use the test function from the proof of \cite[Lemma 2.1]{DV09}.
For $\varepsilon>0$, consider the truncation of $u$ at levels $0$ and $\varepsilon$; that is,
\begin{equation}			\label{eq0926tue}
u_\varepsilon:=\max(\min(u,\varepsilon),0).
\end{equation}
Then, using $u_\varepsilon\geq 0$ as a test function, we obtain
\[
\int_\Omega \mathbf A\nabla u \cdot \nabla u_\varepsilon+\vec b u\cdot \nabla u_\varepsilon+\vec c u_\varepsilon \cdot \nabla u +du u_\varepsilon\leq 0.
\]
Since $uu_\varepsilon\geq0$, using \eqref{eq1123wed}, we have
\[
\int_\Omega\vec c u_\varepsilon\cdot \nabla u+duu_\varepsilon=\int_\Omega \vec c\cdot \nabla(uu_\varepsilon)+duu_\varepsilon-\vec c\nabla u_\varepsilon\cdot u\geq -\int_\Omega\vec c u \cdot \nabla u_\varepsilon.
\]
Using the fact that $\nabla u_\varepsilon$ is supported in the set $\set{0<u<\varepsilon}$ and that $\nabla u = \nabla u_\varepsilon$ there, we derive from the above inequalities that
\[
\lambda\int_{\set{0<u<\varepsilon}} \abs{\nabla u}^2\leq  \int_\Omega \mathbf A \nabla u \cdot \nabla u_\varepsilon \le
-\int_\Omega (\vec b-\vec c) u \cdot \nabla u_\varepsilon \le 
\int_{\set{0<u<\varepsilon}} \abs{\vec b-\vec c}\, \varepsilon\abs{\nabla u}.
\]
So, if we set $N(\varepsilon)=\norm{\vec b-\vec c}_{L^2(\set{0<u<\varepsilon})}$ and $v_\varepsilon=u_\varepsilon-\varepsilon$, we have
\begin{equation}\label{eq:NormOfGrad}
\norm{\nabla v_\varepsilon}_{L^2(\set{0<u<\varepsilon})}\leq \lambda^{-1}\varepsilon N(\varepsilon).
\end{equation}
The rest of the proof is the same argument as in \cite[Lemma 2.1]{DV09}.
If $u\geq\eta>0$ in a set $E$ of positive measure, then $v_\varepsilon=0$ in $E$ whenever $\varepsilon<\eta$ and thus we have
\[
\norm{v_\varepsilon}_{L^2(\Omega) }\leq C_0 \norm{\nabla v_\varepsilon}_{L^2(\Omega)},
\]
for some positive constant $C_0$.
Therefore, since $\abs{v_\varepsilon}> \varepsilon/2$ if $u<\varepsilon/2$, we estimate
\[
\frac{\varepsilon}{2}\, \abs{\set{u<\varepsilon/2}}^{\frac12}\leq\left(\int_{\set{u<\varepsilon/2}} \abs{v_\varepsilon}^2\right)^{\frac12}\leq \norm{v_\varepsilon}_{L^2(\Omega)}\leq C_0 \lambda^{-1}\varepsilon N(\varepsilon).
\]
By dividing by $\varepsilon$ in the above and taking $\varepsilon\to 0$, we see that $\abs{\set{u\leq 0}}=0$.
Therefore, $u>0$ almost everywhere in $\Omega$.
\end{proof}

\begin{lemma}\label{subsol2}
Let $\Omega\subset \bR^n$ be a bounded Lipschitz domain.
Let $\mathbf A=(a^{ij})$ satisfy the uniform ellipticity and boundedness condition \eqref{ellipticity}, $\vec b$, $\vec c \in L^n(\Omega)$, $d \in L^{n/2}(\Omega)$.
Assume that the pair $(\vec b, d)$ satisfies the condition \eqref{eq1122wed}.
Suppose $u\in Y_1^2(\Omega)$ is a subsolution to the problem \eqref{eq0929tue}.
Then, either $u\leq 0$ in $\Omega$ a.e., or $u$ is equal to a positive constant a.e.
\end{lemma}
\begin{proof}
Suppose that $\abs{\set{u>0}}>0$ and $u$ is not equal to a positive constant, then
\[
m=\essinf_\Omega u^{+}<\esssup_\Omega u^{+}=M.
\]
Let $s\in(m,M)$.
Then, since $s>0$, $v=u-s$ is still a subsolution to the same Neumann problem.
Moreover, the sets $D=\set{v>0}$ and $E=\set{v\leq 0}$ both have positive measure.

Now, we apply part (a) of Lemma~\ref{Main} to $v$.
Let $\set{v_i}_{i=1}^N$ be the splitting of $v^+$.
By property (d) of Lemma~\ref{Split}, we find that $v_i=0$ on $E$ for each $i=1, \ldots, N$.
Therefore, by Lemma~\ref{AmpleZero}, we may take $a_i=0$ in \eqref{eq:relation}, and thus by \eqref{eq:MainEstimate}, we have
\[
\int_\Omega \abs{\nabla v^{+}}^2 = 0,
\]
which implies that $v^{+}$ is a constant.
But, since $v^{+}=0$ in $E$ and $\abs{E}>0$, this implies that $v^{+}\equiv 0$ in $\Omega$, which contradicts the fact that $v^{+}>0$ in $D$.
\end{proof}

Now, we can determine the dimension of the kernel of the Neumann problem.

\begin{proposition}\label{kernel}
Let $\Omega\subset \bR^n$ be a bounded Lipschitz domain.
Let $\mathbf A=(a^{ij})$ satisfy the uniform ellipticity and boundedness condition \eqref{ellipticity}, $\vec b$, $\vec c \in L^n(\Omega)$, $d \in L^{n/2}(\Omega)$.
Assume that the pair $(\vec b, d)$ satisfies the condition \eqref{eq1122wed} or the pair $(\vec c,d)$ satisfies the condition \eqref{eq1123wed}.
Then, any nonzero solution $u\in Y_1^2(\Omega)$ to the problem \eqref{eq0929tue} is either almost everywhere positive, or almost everywhere negative.
In particular, the solution space to the problem \eqref{eq0929tue} is at most one dimensional.
\end{proposition}
\begin{proof}
Suppose $(\vec c,d)$ satisfies the condition \eqref{eq1123wed}.
As in \cite[Lemma 2.1]{DV09}, let $u$ be a nonzero solution to the Neumann problem.
If $u>0$ in a set of positive measure, then Lemma~\ref{subsol} shows that $u>0$ almost everywhere in $\Omega$. On the other hand, if $u\leq 0$, then $-u$ is a nonnegative, nonzero solution to the same problem, so again Lemma~\ref{subsol} shows that $-u>0$ almost everywhere.

If $(\vec b,d)$ satisfis the condition~\ref{eq1122wed}, the proof is similar, using Lemma~\ref{subsol2}.
\end{proof}

Since we will assume that either $(\vec b,d)$ satisfies \eqref{eq1122wed} or $(\vec c,d)$ satisfies \eqref{eq1123wed}, by taking $\phi \in C^\infty_c(\bR^n)$ that equals to $1$ in $\Omega$, we see that $d$ must satisfy
\[
\int_\Omega d \ge 0.
\]
We will show that the Neumann problem has a unique solution in the case when $\int_\Omega d>0$, while in the case when $\int_\Omega d =0$, the existence and uniqueness are covered by \cite{DV09} and there is a one-dimensional kernel of solutions to the homogeneous Neumann problem with zero data.
This is the context of the following proposition.

\begin{proposition}			\label{kernelChar}
Let $\Omega$ be a bounded Lipschitz domain.
Let $\mathbf A=(a^{ij})$ satisfy the uniform ellipticity and boundedness condition \eqref{ellipticity}, $\vec b$, $\vec c \in L^n(\Omega)$, $d \in L^{n/2}(\Omega)$.
Assume that the pair $(\vec b, d)$ satisfies the condition \eqref{eq1122wed} or the pair $(\vec c,d)$ satisfies the condition \eqref{eq1123wed}.
Then the dimension of the solution space to the problem \eqref{eq0929tue} is equal to $1$ if and only if $\int_\Omega d=0$, and it is equal to $0$ if and only if $\int_\Omega d>0$.
\end{proposition}
\begin{proof}
First, let us consider the case when $(\vec c,d)$ satisfy \eqref{eq1123wed}.
Suppose $\int_\Omega d=0$.
Let $\psi\in C_c^{\infty}(\bR^n)$ and  set $M=\max\psi^{+}$.
By taking $\phi= \psi^{+}$ and $\phi=M-\psi^{+}$, respectively in \eqref{eq1122wed} and using $\int_\Omega d=0$,  we obtain
\[
\int_\Omega\vec c \cdot \nabla\psi^{+}+d\psi^{+}=0.
\]
The same is true with $\psi^{-}$ in place of $\psi^{+}$ and thus we conclude that
\[
\int_\Omega \vec c\cdot \nabla\psi+d\psi=0,\quad \forall \psi\in C_c^{\infty}(\bR^n).
\]
This implies that $u$ is a solution to the problem \eqref{eq0929tue} if and only if it solves the problem
\begin{equation}				\label{eq1241tue}
\left\{\begin{array}{c l}-\dv(\mathbf{A}\nabla u+(\vec b-\vec c)u)=0 & \text{in }\;\Omega\\
(\mathbf{A}\nabla u+(\vec b-\vec c)u)\cdot \nu=0 & \text{on }\;\partial\Omega.\end{array}\right.
\end{equation}
By \cite[Proposition 2.2]{DV09}, we know that the solution space for the above problem has dimension $1$.
Now, suppose the dimension of the solution space to the problem \eqref{eq0929tue} is nonzero, then it follows from Proposition~\ref{kernel} that there exists a solution $u$ with $u>0$ almost everywhere in $\Omega$.
Set $u_\varepsilon = \min(u, \epsilon)$ so that $0< u_\varepsilon \le \varepsilon$.
Using $u-u_\varepsilon\geq 0$ as a test function in \eqref{eq1123wed}, we have
\[
0\leq\int_\Omega\vec c \cdot \nabla(u-u_\varepsilon)+d(u-u_\varepsilon)=-\int_\Omega\vec c \cdot \nabla u_\varepsilon+du_\varepsilon,
\]
where we used
\[
\int_\Omega\vec c\cdot \nabla u+du=0,
\]
which follows from using $1$ as a test function to the problem \eqref{eq0929tue}.
On the other hand, by using $u_\varepsilon\ge 0$, as a test function in \eqref{eq1123wed}, we have
\[
\int_\Omega\vec c \cdot \nabla u_\varepsilon+du_\varepsilon \ge 0,
\]
and thus we find that
\[
\int_\Omega\vec c \cdot \nabla u_\varepsilon+du_\varepsilon = 0.
\]
Set $v_\varepsilon=u_\varepsilon - \varepsilon$.
Note that, as in \eqref{eq:NormOfGrad}, we have
\[
\norm{\nabla v_\varepsilon}_{L^2(\set{0<u<\varepsilon})}\leq \lambda^{-1}\varepsilon \norm{\vec b-\vec c}_{L^2(\set{0<u<\varepsilon})}.
\]
Therefore, we have
\begin{align*}
\Abs{\int_\Omega d \,\frac{u_\varepsilon}{\varepsilon}}=\Abs{\frac{1}{\varepsilon}\int_\Omega d u_\varepsilon}& \leq\frac{1}{\varepsilon}  \int_\Omega \abs{\vec c\cdot \nabla u_\varepsilon}\\
&=\frac{1}{\varepsilon} \int_{\set{0<u<\epsilon}} \abs{\vec c\cdot \nabla v_\varepsilon} \leq \lambda^{-1} \norm{\vec c}_{L^2(\set{0<u<\epsilon})}\, \norm{\vec b-\vec c}_{L^2(\set{0<u<\varepsilon})}.
\end{align*}
Then, by letting $\varepsilon\to 0$ in the above and noting that $u_\varepsilon / \varepsilon \to 1$, we obtain $\int_\Omega d=0$.

Next, we consider the case when $(\vec b,d)$ satisfy Condition~\eqref{eq1122wed}.
In this case, the proof follows from the Fredholm alternative by observing that for the adjoint equation the role of $\vec b$ and $\vec c$ is reversed.
\end{proof}

Having found the dimension of the kernels to the homogeneous problem, we turn to existence of solutions.
As we have observed in the proof of Proposition~\ref{kernelChar}, in the case when $\int_\Omega d=0$, $u$ is a solution of \eqref{eq0929tue} if and only if it is a solution of \eqref{eq1241tue}.
Therefore, the case when $\int_\Omega d=0$ is covered by \cite[Theorem~1.1]{DV09}, which is the following.

\begin{proposition}\label{ExistenceDelta=0}
Let $\Omega$ be a bounded Lipschitz domain.
Let $\mathbf A=(a^{ij})$ satisfy the uniform ellipticity and boundedness condition \eqref{ellipticity}, and $\vec b \in L^n(\Omega)$.
There exists $\widehat{u}\in Y_1^2(\Omega)$ with $\hat{u}>0$ almost everywhere and $\norm{\widehat u}_{L^{\frac{2n}{n-2}}(\Omega)}=1$ such that $u\in Y_1^2(\Omega)$ is a solution to the Neumann problem
\[
\left\{\begin{array}{c l}-\dv(\mathbf A\nabla u+\vec bu)=0 & \text{in }\;\Omega,\\
(\mathbf{A} \nabla u+\vec b u)\cdot \nu=0& \text{on }\;\partial\Omega\end{array}\right.
\]
if and only if $u=c\widehat{u}$ for some constant $c\in\bR$.
Moreover, for any $f\in L^{\frac{2n}{n+2}}(\Omega)$, $\vec F\in L^2(\Omega)$ and $g\in L^{2-\frac{2}{n}}(\partial\Omega)$ satisfying the compatibility condition
\begin{equation}\label{eq:Comp1}
\int_\Omega f+\int_{\partial\Omega} g=0,
\end{equation}
there exists a unique solution $u\in Y_1^2(\Omega)$ to the Neumann problem
\[
\left\{\begin{array}{c l}-\dv(\mathbf A \nabla u+\vec b u)=f-\dv \vec F & \text{in }\;\Omega\\
(\mathbf A \nabla u +\vec b u)\cdot \nu=g+\vec F\cdot \nu & \text{on }\;\partial\Omega\end{array}\right.
\]
with $\int_\Omega u=0$.
\end{proposition}
\begin{proof}
See Theorem~1.1 in \cite{DV09}.
\end{proof}

The Fredholm alternative yields the following result for the adjoint equation.

\begin{proposition}		\label{ExistenceDelta=0Adj}
Let $\Omega$ be a bounded Lipschitz domain.
Let $\mathbf A=(a^{ij})$ satisfy the uniform ellipticity and boundedness condition \eqref{ellipticity}, and $\vec b \in L^n(\Omega)$.
Let $\widehat{u}$ be as in Proposition \ref{ExistenceDelta=0}.
Assume that $f\in L^{\frac{2n}{n+2}}(\Omega)$, $\vec F\in L^2(\Omega)$, $g\in L^{2-\frac{2}{n}}(\partial\Omega)$ satisfy the compatibility condition
\begin{equation}\label{eq:Comp2}
\int_\Omega(f\widehat{u}+\vec F\cdot \nabla\widehat{u})+\int_{\partial\Omega}g\widehat{u}=0.
\end{equation}
Then, there exists a unique solution $u\in Y_1^2(\Omega)$ to the problem
\[
\left\{\begin{array}{c l}-\dv(\mathbf A\tran\nabla u)+\vec b \cdot \nabla u=f-\dv \vec F& \text{in }\;\Omega,\\ \mathbf A\tran \nabla u\cdot \nu=g+\vec F\cdot \nu & \text{on}\,\,\,\partial\Omega,\end{array}\right.
\]
satisfying $\int_\Omega u=0$.
\end{proposition}

In the case $\int_\Omega d>0$, we have the following result.

\begin{proposition}\label{ExistenceDelta>0}
Let $\Omega$ be a bounded Lipschitz domain.
Let $\mathbf A=(a^{ij})$ satisfy the uniform ellipticity and boundedness condition \eqref{ellipticity}, and $\vec b$, $\vec c \in L^n(\Omega)$, and $d\in L^{\frac{n}{2}}(\Omega)$.
Suppose $(\vec b, d)$ satisfies \eqref{eq1122wed} or $(\vec c,d)$ satisfies \eqref{eq1123wed}.
Assume also that $\int_\Omega d>0$.
Then, for any $f\in L^{\frac{2n}{n+2}}(\Omega)$, $\vec F\in L^2(\Omega)$ and $g\in L^{2-\frac{2}{n}}(\Omega)$, there exists a unique solution $u\in Y_1^2(\Omega)$ to the Neumann problem
\[
\left\{\begin{array}{c l}
-\dv(\mathbf A \nabla u+\vec bu)+\vec c \cdot \nabla u+du = f-\dv \vec F &\text{ in }\;\Omega,\\
(\mathbf A \nabla u + \vec b u)\cdot \nu = g+ \vec F \cdot \nu & \text{ on }\;\partial\Omega.
\end{array}\right.
\]
\end{proposition}
\begin{proof}
The proof is similar to the proof of Theorem~1.1 in \cite{DV09}, but in the case when $\int_\Omega d >0$, Proposition~\ref{kernelChar} says that the only $W_1^2(\Omega)$ solution to the homogeneous Neumann problem \eqref{eq0929tue} is $u=0$.
\end{proof}

%------------------------------------------------------------------------------%
\section{Estimates for solutions}			\label{sec4}
%------------------------------------------------------------------------------%
In this section, we will establish the global estimates for solutions in the Sobolev space $Y_1^2(\Omega)$.
We consider the case when $\int_\Omega d>0$ and $\int_\Omega d=0$ separately.
%------------------------------------------------------------------------------%
\subsection{The case when $\boldsymbol{\int_\Omega d=0}$}
%------------------------------------------------------------------------------%
In the proof of Proposition~\ref{kernelChar}, we saw that if $(\vec b, d)$ satisfies \eqref{eq1122wed} and $\int_\Omega d=0$, then the problem \eqref{eq0929tue} reduces to \eqref{eq1241tue} and it is enough to consider the reduced operator
\[
L_0 u=-\dv(\mathbf A \nabla u + \vec b u)
\]
under the assumption that $b\in L^n(\Omega)$.
This question has been studied in \cite{DV09}, and estimates for solutions are provided.
However, in \cite{DV09}, the constants do not depend on the parameters in an optimal way and we clarify this here.
Our first estimate concerns subsolutions. 

\begin{proposition}\label{EstimateForDelta=0}
Let $\Omega\subset \bR^n$ be a bounded Lipschitz domain.
Let $\mathbf A=(a^{ij})$ satisfy the uniform ellipticity and boundedness condition \eqref{ellipticity}, and $\vec b \in L^n(\Omega)$.
Assume that $f \in L^{\frac{2n}{n+2}}(\Omega)$, $\vec F\in L^2(\Omega)$, and $g\in L^{2-\frac{2}{n}}(\partial\Omega)$.
Suppose $u\in Y_1^2(\Omega)$ is a subsolution to the problem
\begin{equation}		\label{eq1826tue}
\left\{\begin{array}{c l}-\dv(\mathbf A\nabla u+\vec bu)= f-\dv \vec F & \text{in }\;\Omega\\
(\mathbf A \nabla u + \vec b u)\cdot \nu = g+\vec F\cdot \nu& \text{on }\;\partial\Omega,\end{array}\right.
\end{equation}
or a subsolution to the problem
\begin{equation}		\label{eq1832tue}
\left\{\begin{array}{c l}-\dv(\mathbf A\nabla u)+\vec b \cdot \nabla u = f-\dv \vec F & \text{in }\; \Omega\\
\mathbf A \nabla u \cdot \nu = g+\vec F\cdot \nu& \text{on }\;\partial\Omega.\end{array}\right.
\end{equation}
Then, we have
\begin{equation}			\label{eq:EstDelta=0}
\norm{u^{+}}_{Y_1^2(\Omega) }\leq C \left( \norm{f^{+}}_{L^{\frac{2n}{n+2}}(\Omega)}+\norm{\vec F}_{L^2(\Omega)}+\norm{g^{+}}_{L^{2-\frac{2}{n}}(\partial\Omega)}+\abs{\Omega}^{-\frac{n+2}{2n}}\int_\Omega u^{+} \right),
\end{equation}
where $C$ depends on $n$, $\lambda$, $\norm{\vec b}_n$, and the Lipschitz character of $\Omega$.
\end{proposition}
\begin{proof}
Let $u$ be a subsolution of the problem \eqref{eq1826tue}.
Let $r=\abs{\Omega}^{\frac{1}{n}}$ and $\Omega_r=\frac{1}{r}\Omega$.
For $x\in\Omega_r$, set $u_r(x):=u(rx)$, and etc.
Then, we have $\abs{\Omega_r}=1$ and it is straightforward to see that $u_r$ is a subsolution to the problem
\[
\left\{\begin{array}{c l}-\dv(\mathbf A_r\nabla u_r+r \vec b_r u_r)= r^2 f_r-\dv r\vec F_r, & \text{in }\; \Omega_r\\
(\mathbf A_r \nabla u_r + r \vec b_r)\cdot \nu = r g_r+r\vec F_r \cdot \nu & \text{on }\;\partial\Omega_r\end{array}\right.
\]
Moreover, $\Omega_r$ is still a Lipschitz domain with the same Lipschitz character as $\Omega$.
Therefore, once we establish the estimate in this case, that is,
\[
\norm{u_r^{+}}_{Y_1^2(\Omega_r)} \leq C\left(\norm{r^2 f_r^{+}}_{L^{\frac{2n}{n+2}}(\Omega_r)}+C\norm{r \vec F_r}_{L^2(\Omega_r)}+C\norm{r g_r^{+}}_{L^{2-\frac{2}{n}}(\partial\Omega_r)}+C\int_{\Omega_r}u_r^{+}\right),
\]
then rescaling back to $\Omega$ will yield the estimate we want.
Hence, we may and shall assume that $\abs{\Omega}=1$.
Then, the estimate \eqref{eq:EstDelta=0} follows from Lemma~\ref{Main}, by choosing $a_i=\fint_\Omega  u_i \ge 0$, taking $C_0 \ge 1$ to be the constant in Sobolev-Poincar\'e inequality, and taking $\varepsilon=\frac{\lambda}{8C_0}$.

The proof for case when $u$ is a subsolution of the problem \eqref{eq1832tue} is  parallel.
\end{proof}

As an application, we bound the norm of the function $\hat{u}$ in Proposition~\ref{ExistenceDelta=0}.

\begin{lemma}
Let $\Omega\subset \bR^n$ be a bounded Lipschitz domain.
Let $\mathbf A=(a^{ij})$ satisfy the uniform ellipticity and boundedness condition \eqref{ellipticity}, and $\vec b \in L^n(\Omega)$.
Then, the function $\widehat{u}\in Y_1^2(\Omega)$ in Proposition~\ref{ExistenceDelta=0} satisfies the estimate
\[
\norm{\widehat u}_{Y_1^2(\Omega)}\leq C,
\]
where $C$ depends on $n$, $\lambda$, $\norm{\vec b}_n$, and the Lipschitz character of $\Omega$.
\end{lemma}
\begin{proof}
Recall that $\widehat{u}>0$ in $\Omega$ and $\norm{\widehat u}_{L^{\frac{2n}{n-2}}(\Omega)}=1$.
Therefore, the lemma follows from Proposition~\ref{EstimateForDelta=0}.
\end{proof}

We now turn to the estimate for solutions.
\begin{proposition}	\label{SolvabilityDelta=0}
Let $\Omega\subset \bR^n$ be a bounded Lipschitz domain.
Assume that $\mathbf A=(a^{ij})$ satisfies the uniform ellipticity and boundedness condition \eqref{ellipticity}, and $\vec b \in L^n(\Omega)$.
Let $f\in L^{\frac{2n}{n+2}}(\Omega)$, $\vec F\in L^2(\Omega)$, $g\in L^{2-\frac{2}{n}}(\partial\Omega)$, and assume that the compatibility condition \eqref{eq:Comp1} holds.
Then there exists a unique solution $u\in Y_1^2(\Omega)$ to the problem \eqref{eq1826tue} satisfying $\int_\Omega u=0$.
Moreover,
\begin{equation}\label{eq:SolveDelta=0}
\norm{u}_{Y_1^2(\Omega)}\leq C \left( \norm{\vec F}_{L^2(\Omega)}+\norm{f}_{L^{\frac{2n}{n+2}}(\Omega)}+ \norm{g}_{L^{2-\frac{2}{n}}(\partial\Omega)}\right),
\end{equation}
where $C$ depends on $n$, $\lambda$, $\norm{\vec b}_n$, and the Lipschitz character of $\Omega$.
\end{proposition}
\begin{proof}
Existence follows from Proposition~\ref{ExistenceDelta=0}.
To show the estimate, as in the proof of Proposition~\ref{EstimateForDelta=0}, we may assume that $\abs{\Omega}=1$.
For $\varepsilon>0$ small enough, depending on $n$, $\lambda$, and the Lipschitz character of $\Omega$, consider the splitting $\set{u_i}_{i=1}^N$ of $u$ subject to $h=\abs{\vec b}$ and $\varepsilon$, which is constructed in Lemma~\ref{SplitMeanZero}.
Since $\int_\Omega u_i =0$, the Sobolev-Poincar\'e inequality implies that there is a constant $C_0$ such that
\[
\norm{u_i}_{L^{\frac{2n}{n-2}}(\Omega)}\leq C_0 \norm{\nabla u_i}_2.
\]
Therefore, we can take $a_i=0$ in part (b) of Lemma~\ref{Main}.
\end{proof}

\begin{remark}
It is interesting to note that the combination of Proposition~\ref{EstimateForDelta=0} for subsolutions with its analogue for supersolutions are not enough to show the estimate for solutions in Proposition~\ref{SolvabilityDelta=0}.
This comes from the fact that the subsolution and the supersolution estimate only ``see" the positive and negative parts of $u$, respectively.
However, Proposition~\ref{SolvabilityDelta=0} relies on the cancellation condition $\int_\Omega u=0$.
\end{remark}

The adjoint equation is treated in a similar way.
\begin{proposition}			\label{prop2132}
Let $\Omega\subset \bR^n$ be a bounded Lipschitz domain.
Assume that $\mathbf A=(a^{ij})$ satisfies the uniform ellipticity and boundedness condition \eqref{ellipticity}, and $\vec b \in L^n(\Omega)$.
Let $f\in L^{\frac{2n}{n+2}}(\Omega)$, $\vec F\in L^2(\Omega)$, $g\in L^{2-\frac{2}{n}}(\partial\Omega)$, and assume that the compatibility condition  \eqref{eq:Comp2} holds.
Then there exists a unique solution $u\in Y_1^2(\Omega)$ to the problem
\[
\left\{\begin{array}{c l}-\dv(\mathbf A\tran\nabla u)+\vec b \cdot \nabla u=f-\dv \vec F& \text{in }\;\Omega,\\ \mathbf A\tran \nabla u\cdot \nu=g+\vec F\cdot \nu & \text{on}\,\,\,\partial\Omega,\end{array}\right.
\]
satisfying $\int_\Omega u=0$.
Moreover, we have
\begin{equation*}		%\label{eq:SolveDelta=0Adj}
\norm{u}_{Y_1^2(\Omega)}\leq C \left( \norm{\vec F}_{L^2(\Omega)}+\norm{f}_{L^{\frac{2n}{n+2}}(\Omega)}+\norm{g}_{L^{2-\frac{2}{n}}(\partial\Omega)}\right),
\end{equation*}
where $C$ depends on $n$, $\lambda$, $\norm{\vec b}_n$, and the Lipschitz character of $\Omega$.
\end{proposition}

%------------------------------------------------------------------------------%
\subsection{The case when $\boldsymbol{\int_\Omega d>0}$}
%------------------------------------------------------------------------------%
In this case, we obtain existence, uniqueness, as well as scale invariant estimates that also depend on the value $\int_\Omega d$.
The idea to show these estimates is to use Lemma~\ref{Main}, subtracting suitable constants $a_i$ at each step.
The choice of the constants is motivated by the following lemma.

\begin{lemma}\label{ChoiceOfConstants}
Let $\Omega\subset\bR^n$ be a bounded Lipschitz domain.
Let $\mathbf A=(a^{ij})$ satisfy the uniform ellipticity and boundedness condition \eqref{ellipticity}, $\vec b$, $\vec c \in L^n(\Omega)$, $d \in L^{\frac{n}{2}}(\Omega)$.
Let $f\in L^{\frac{2n}{n+2}}(\Omega)$ and $g\in L^{2-\frac{2}{n}}(\partial\Omega)$.
If $u$ is a subsolution to the Neumann problem
\[
\left\{\begin{array}{c l}
-\dv(\mathbf A\nabla u+\vec b u)+\vec c\cdot\nabla u+du=f & \text{in }\;\Omega,\\
(\mathbf A \vec u+ \vec b u)\cdot \nu= g & \text{on }\;\partial\Omega,
\end{array}\right.
\]
then
\begin{equation}\label{eq:avgEst}
\int_\Omega \left(\vec c \cdot \nabla u^{+} + du^{+} \right) \leq\int_\Omega f^{+}+\int_{\partial\Omega} g^{+}.
\end{equation}
\end{lemma}
\begin{proof}
For $\varepsilon>0$, let $u_\varepsilon$ be the truncation of $u$ at levels $0$ and $\varepsilon$ as in \eqref{eq0926tue}.
Using $u_\varepsilon$ as a test function, we obtain
\[
\int_\Omega \mathbf A\nabla u \cdot \nabla u_\varepsilon+\vec b u \cdot \nabla u_\varepsilon+ \vec c u_\varepsilon \cdot \nabla u+du u_\varepsilon\leq\int_\Omega f u_\varepsilon+\int_{\partial\Omega} gu_\varepsilon.
\]
Since $0 \le u_\varepsilon \le \epsilon$ and $\nabla u_\varepsilon$ is supported in the set $\set{0<u<\varepsilon}$, we have
\begin{align*}
\lambda \int_\Omega \abs{\nabla u_\varepsilon}^2+\int_\Omega \vec c u_\varepsilon \cdot \nabla u+du u_\varepsilon &\leq \varepsilon \int_\Omega f^{+}+\varepsilon \int_{\partial\Omega} g^{+}+\varepsilon \norm{\vec b}_{L^2(\set{0<u<\varepsilon})}\, \norm{\nabla u_\varepsilon}_2\\
&\leq \varepsilon\int_\Omega f^{+}+\varepsilon\int_{\partial\Omega} g^{+}+\varepsilon \norm{\vec b}_{L^2(\set{0<u<\varepsilon})}^2+\frac{\varepsilon}{4}\, \norm{\nabla u_\varepsilon}_2^2.
\end{align*}
Therefore, if $\varepsilon < 4\lambda$, dividing the above inequality by $\varepsilon$, we obtain
\begin{align*}
\int_\Omega (\vec c \cdot \nabla u+du) \,\frac{u_\varepsilon}{\varepsilon} &\leq \int_\Omega f^{+}+\int_{\partial\Omega} g^{+}+\norm{\vec b}_{L^2(\set{0<u<\varepsilon})}^2.
\end{align*}
The proof is complete since $u_\varepsilon/\varepsilon \to \chi_{\set{u>0}}$ and $\norm{\vec b}_{L^2(\set{0<u<\varepsilon})}\to 0$ as $\varepsilon \to 0$.
\end{proof}

A Poincar\'e-type estimate is the context of the following lemma.

\begin{lemma}\label{ModifiedPoincare}
Let $\Omega\subset \bR^n$ be a Lipschitz domain with $\abs{\Omega}=1$.
Assume that $\vec c\in L^n(\Omega)$, $d\in L^{\frac{n}{2}}(\Omega)$, and $\int_\Omega d\geq\delta_0>0$.
Then, there exists a constant $C_0>0$ such that for every $u\in W_1^2(\Omega)$, we have
\[
\Norm{u-\frac{\int_\Omega (\vec c \cdot \nabla u+du)}{\int_\Omega d}}_{L^{\frac{2n}{n-2}}(\Omega)} \leq C_0 \norm{\nabla u}_{L^2(\Omega)}.
\]
This constant $C_0$ depends only on $n$, $\norm{\vec c}_n$, $\norm{d}_{n/2}$, $\delta_0$, and the Lipschitz character of $\Omega$.
\end{lemma}
\begin{proof}
Note that
\begin{align*}
\Abs{\frac{\int_\Omega(\vec c \cdot \nabla u+du)}{\int_\Omega d}-\fint_\Omega u} &=\frac{1}{\int_\Omega d}\,\Abs{\int_\Omega \vec c \cdot \nabla u+d\left(u-\fint_\Omega u\right)dx}\\
&\leq \frac{1}{\delta_0} \left(\norm{\vec c}_{L^2(\Omega)} \norm{\nabla u}_{L^2(\Omega)} +\norm{d}_{L^{\frac{2n}{n+2}}(\Omega)}\Norm{u-\fint_\Omega u}_{L^{\frac{2n}{n-2}}(\Omega)}\right).
\end{align*}
Also, we have
\[
\Norm{u-\frac{\int_\Omega(\vec c \cdot \nabla u+du)}{\int_\Omega d}}_{L^{\frac{2n}{n-2}}(\Omega)} \le
\Norm{u-\fint_\Omega u}_{L^{\frac{2n}{n-2}}(\Omega)}+\Norm{\fint_\Omega u-\frac{\int_\Omega(\vec c \cdot \nabla u+du)}{\int_\Omega d}}_{L^{\frac{2n}{n-2}}(\Omega)}.
\]
The proof is complete by combining these estimates, applying the Sobolev-Poincar\'e inequality and H\"older's inequality, and using the assumption that $\abs{\Omega}=1$.
\end{proof}

We now turn to the following estimate for and solutions and subsolutions.

\begin{proposition}\label{EstimateForDelta>0}
Let $\Omega\subset\bR^n$ be a Lipschitz domain.
Let $\mathbf A=(a^{ij})$ satisfy the uniform ellipticity and boundedness condition \eqref{ellipticity}, $\vec b$, $\vec c \in L^n(\Omega)$, $d \in L^{n/2}(\Omega)$.
Assume that the pair $(\vec b, d)$ satisfies the condition \eqref{eq1122wed} or the pair $(\vec c,d)$ satisfies the condition \eqref{eq1123wed}.
Assume that  $\abs{\Omega}^{\frac{2}{n}-1}\int_\Omega d \geq \delta_0>0$.
Let $f\in L^{\frac{2n}{n+2}}(\Omega)$, $\vec F\in L^2(\Omega)$, $g\in L^{2-\frac{2}{n}}(\partial\Omega)$.
\begin{enumerate}
\item 
If $u\in Y_1^2(\Omega)$ is a subsolution to the problem
\begin{equation}				\label{eq1540fri}
\left\{\begin{array}{c l}
-\dv(\mathbf A \nabla u+\vec bu)+\vec c \cdot \nabla u+du = f-\dv \vec F &\text{ in }\;\Omega,\\
(\mathbf A \nabla u + \vec b u)\cdot \nu = g+ \vec F \cdot \nu & \text{ on }\;\partial\Omega,
\end{array}\right.
\end{equation}
then there exists $C>0$, depending on $n$, $\lambda$, $\norm{\vec b}_n$, $\norm{\vec c}_n$, $\norm{d}_{n/2}$, $\delta_0$, and the Lipschitz character of $\Omega$, such that
\begin{equation}			\label{eq1235wed}
\norm{u^{+}}_{Y_1^2(\Omega)}\leq C\left( \norm{f^{+}}_{L^{\frac{2n}{n+2}}(\Omega)}+\norm{\vec F}_{L^2(\Omega)}+\norm{g^{+}}_{L^{2-\frac{2}{n}}(\partial\Omega)}\right).
\end{equation}
\item
There exists a unique solution $u\in Y_1^2(\Omega)$ to the problem \eqref{eq1540fri} and it satisfies
\begin{equation}			\label{eq1608fri}
\norm{u}_{Y_1^2(\Omega)}\le C\left( \norm{f}_{L^{\frac{2n}{n+2}}(\Omega)}+\norm{\vec F}_{L^2(\Omega)}+\norm{g}_{L^{2-\frac{2}{n}}(\partial\Omega)}\right),
\end{equation}
where  $C$ depends on $n$, $\lambda$, $\norm{\vec b}_n$, $\norm{\vec c}_n$, $\norm{d}_{n/2}$, $\delta_0$, and the Lipschitz character of $\Omega$.
\end{enumerate}
\end{proposition}

\begin{proof}
As in the proof of Proposition~\ref{SolvabilityDelta=0}, take $r=\abs{\Omega}^{\frac{1}{n}}$, let $\Omega_r=\frac{1}{r}\Omega$, and for $x\in\Omega_r$, set $u_r(x)=u(rx)$, etc.
Then, $u_r$ becomes a subsolution to the problem
\[
\left\{\begin{array}{c l}-\dv(\mathbf A_r \nabla u_r+ r\vec b_r u_r)+r\vec c_r\cdot \nabla u_r+r^2 d_r u_r = r^2 f_r- \dv r\vec F_r & \text{in }\;\Omega_r,\\
(\mathbf A_r \nabla u_r + r \vec b_r)\cdot \nu = r g_r+r\vec F_r \cdot \nu& \text{on }\;\partial\Omega_r\end{array}\right.
\]
Since $r=\abs{\Omega}^{\frac{1}{n}}$, we have $\abs{\Omega_r}=1$, the norms of the lower order coefficients are preserved, and
\[
\int_{\Omega_r} r^2 d_r(y)\,dy=\int_{\Omega_r}r^2 d(ry)\,dy=r^{2-n}\int_\Omega d \geq\delta_0.
\]
Therefore, as in the proof of Proposition~\ref{SolvabilityDelta=0}, we may and will assume that $\abs{\Omega}=1$.

First, we treat the case when $(\vec c,d)$ satisfies \eqref{eq1123wed}.
Also, let us momentarily assume that $\vec F= 0$.
We apply Lemma~\ref{Split} to $u^{+}$, with $h=\abs{\vec b-\vec c}$ and $\varepsilon=\frac{\lambda}{2C_0}$, where $C_0$ is as in Lemma~\ref{ModifiedPoincare}, to obtain the splitting $\set{u_i}_{i=1}^N$ of $u^{+}$. Note that $u_i \ge 0$.
Then, by setting
\[
a_i:=\frac{1}{\int_\Omega d}\int_\Omega (\vec c \cdot \nabla u_i+du_i) \ge 0,
\]
and applying Lemma~\ref{ModifiedPoincare}, we have
\[
\norm{u_i - a_i}_{\frac{2n}{n-2}} \le C_0 \norm{\nabla u_i}_2.
\]
Also, note that we have
\[
a=\sum_{i=1}^N a_i =\frac{1}{\int_\Omega d}\sum_{i=1}^N \int_\Omega (\vec c \cdot \nabla u_i+du_i)=\frac{1}{\int_\Omega d}\int_\Omega (\vec c \cdot \nabla u^{+}+du^{+}).
\]
Therefore, the estimate \eqref{eq1235wed} follows from Lemma~\ref{Main} and \eqref{eq:avgEst}.

Let us now treat the case when $\vec F$ is not identically zero.
By Proposition~\ref{SolvabilityDelta=0}, there exists a function $\tilde u\in Y_1^2(\Omega)$ with  $\int_\Omega \tilde u=0$ that solves
\begin{equation}				\label{eq1655fri}
\left\{\begin{array}{c l}-\dv(\mathbf A\nabla \tilde u+\vec b \tilde u)=-\dv \vec F& \text{in }\;\Omega,\\
(\mathbf A \nabla \tilde u+\vec b \tilde u)\cdot \nu=\vec F\cdot \nu & \text{on }\;\partial\Omega.
\end{array}\right.
\end{equation}
Then, $w=u-\tilde u$ becomes a subsolution to the problem
\begin{equation}				\label{eq1700fri}
\left\{\begin{array}{c l}-\dv(\mathbf A\nabla w+\vec bw)+\vec c\cdot \nabla w+dw = f-\vec c\cdot \nabla \tilde u-d\tilde u & \text{in }\;\Omega,\\
(\mathbf A \nabla w + \vec b w)\cdot \nu = g & \text{on }\;\partial\Omega.\end{array}\right.
\end{equation}
Since $0\leq u^{+} \leq(u-\tilde u)^{+}+ \tilde u^{+} \leq w^{+}+ \abs{\tilde u}$, the estimate when $\vec F= 0$ yields
\begin{align*}
\norm{u^{+}}_{L^{\frac{2n}{n-2}}(\Omega)} &\leq \norm{w^{+}}_{L^{\frac{2n}{n+2}}(\Omega)}+\norm{\tilde u}_{L^{\frac{2n}{n-2}}(\Omega)}\\
&\leq C\norm{(f-\vec c\cdot \nabla \tilde u-d\tilde u)^{+}}_{L^{\frac{2n}{n+2}}(\Omega)}+\norm{g^{+}}_{L^{2-\frac{2}{n}}(\partial\Omega)}+\norm{\tilde u}_{Y_1^2(\Omega)},
\end{align*}
Hence, H\"older's inequality combined with estimate \eqref{eq:SolveDelta=0} for $\norm{\tilde u}_{Y_1^2(\Omega)}$ shows that
\begin{equation}\label{eq:2*}
\norm{u^{+}}_{L^{\frac{2n}{n-2}}(\Omega) }\leq C \left(\norm{f^{+}}_{L^{\frac{2n}{n+2}}(\Omega)}+\norm{\vec F}_{L^2(\Omega)}+\norm{g^{+}}_{L^{2-\frac{2}{n}}(\partial\Omega)}\right).
\end{equation}
To bound the $L^2$ norm of $\nabla u^{+}$, use $u^{+}$ as a test function, apply the condition \eqref{eq1123wed}, and get
\[
\int_\Omega \mathbf A \nabla u^{+} \cdot \nabla u^{+} \leq \Abs{\int_\Omega (\vec b-\vec c)u^{+} \cdot \nabla u^{+} } +\int_\Omega fu^{+}+\int_\Omega \vec F\cdot \nabla u^{+}+\int_{\partial\Omega}gu^{+}.
\]
Then by using the ellipticity \eqref{ellipticity}, H\"older's inequality, the trace inequality, Cauchy's inequality, and \eqref{eq:2*}, we get the desired estimate \eqref{eq1235wed}.

Now, let us turn to the proof for part (b).
Existence and uniqueness of the solution $u$ to the problem \eqref{eq1540fri} is given in Proposition~\ref{ExistenceDelta>0}.
The estimate \eqref{eq1608fri} is obtained by applying Proposition~\ref{EstimateForDelta>0} to $u$ and $-u$.

Next, let us treat the case when $(\vec b,d)$ satisfies \eqref{eq1122wed}.
We will prove part (b) first by using a duality argument.
Again, we momentarily assume that $\vec F=0$ and consider the problem
\begin{equation}			\label{eq1638fri}
\left\{\begin{array}{c l}
-\dv(\mathbf A \nabla u+\vec bu)+\vec c \cdot \nabla u+du = f &\text{ in }\;\Omega,\\
(\mathbf A \nabla u + \vec b u)\cdot \nu = g & \text{ on }\;\partial\Omega.
\end{array}\right.
\end{equation}
By Proposition~\ref{ExistenceDelta>0}, there is a unique solution $u \in Y_1^2(\Omega)$ of the problem.
For an arbitrary $\tilde f \in L^{\frac{2n}{n+2}}(\Omega)$, let $v \in Y_1^2(\Omega)$ be the solution of the problem
\begin{equation}			\label{eq1639fri}
\left\{\begin{array}{c l}
-\dv(\mathbf A\tran \nabla v+\vec c v)+\vec b \cdot \nabla v+dv = \tilde f &\text{ in }\;\Omega,\\
(\mathbf A \nabla u + \vec b u)\cdot \nu = 0 & \text{ on }\;\partial\Omega.
\end{array}\right.
\end{equation}
Then by the estimate \eqref{eq1608fri} applied to $v$, we have
\begin{equation}			\label{eq1636fri}
\norm{v}_{Y_1^2(\Omega)}\le C \norm{\tilde f}_{L^{\frac{2n}{n+2}}(\Omega)}.
\end{equation}
Since $u$ and $v$ are solutions to \eqref{eq1638fri} and \eqref{eq1639fri}, respectively, we have
\[
\int_\Omega f v + \int_{\partial\Omega} gv = \int_\Omega \tilde f u.
\]
Therefore, by \eqref{eq1636fri} combined with H\"older's inequality and the trace inequality, we obtain
\[
\Abs{\int_\Omega \tilde f u} \le C\left( \norm{f}_{L^{\frac{2n}{n+2}}(\Omega)} + \norm{g}_{L^{2-\frac{2}{n}}(\partial\Omega)}\right) \norm{\tilde f}_{L^{\frac{2n}{n+2}}(\Omega)}.
\]
Since the above estimate holds for any $\tilde f \in L^{\frac{2n}{n+2}}(\Omega)$, the converse to H\"older's inequality yields that
\begin{equation*}%			\label{eq1649fri}
\norm{u}_{L^{\frac{2n}{n-2}}(\Omega)} \le C\left( \norm{f}_{L^{\frac{2n}{n+2}}(\Omega)} + \norm{g}_{L^{2-\frac{2}{n}}(\partial\Omega)}\right).
\end{equation*}
Now, we deal with the case when the $\vec F$ term is present.
Let $u \in Y_1^2(\Omega)$ be the solution to the problem \eqref{eq1540fri}, whose existence and uniqueness is again guranteed by Proposition~\ref{ExistenceDelta>0}.
Also, let $\tilde u$ be a solution of \eqref{eq1655fri} with $\int_\Omega \tilde u=0$, which exists by by Proposition~\ref{SolvabilityDelta=0}.
Then, $w=u-\tilde u$ is becomes a solution to the problem \eqref{eq1700fri}, and hence, similar to \eqref{eq:2*}, we have
\begin{equation}				\label{eq1708fri}
\norm{u}_{L^{\frac{2n}{n-2}}(\Omega) }\leq C \left(\norm{f}_{L^{\frac{2n}{n+2}}(\Omega)}+\norm{\vec F}_{L^2(\Omega)}+\norm{g}_{L^{2-\frac{2}{n}}(\partial\Omega)}\right).
\end{equation}
On the other hand, by testing $u$ to \eqref{eq1540fri} and applying the condition \eqref{eq1122wed}, we have
\[
\int_\Omega \mathbf A \nabla u \cdot \nabla u = \int_\Omega (\vec b-\vec c)u \cdot \nabla u +\int_\Omega fu+\int_\Omega \vec F\cdot \nabla u+\int_{\partial\Omega}gu.
\]
Then by using the ellipticity \eqref{ellipticity}, H\"older's inequality, the trace inequality, and Cauchy's inequality, and \eqref{eq1708fri}, we get the desired estimate \eqref{eq1235wed}.

Finally, let us prove part (a) under the condition \eqref{eq1122wed}.
Suppose $u \in Y_1^2(\Omega)$ is a subsolution of the problem \eqref{eq1540fri}.
Then may assume that $f, g\geq 0$.
Let $u_0 \in Y_1^2(\Omega)$ be the solution of the same problem, which we just investigated.
Then $u-u_0$ becomes a subsolution to the homogeneous problem \eqref{eq0929tue}.
By Lemma~\ref{subsol2}, we see that $(u-u_0)^{+}=0$ a.e. in $\Omega$; otherwise, $1$ becomes a subsolution of \eqref{eq0929tue}, which would imply $\int_\Omega d =0$.
Therefore, we have
\[
0 \le u^+ \le (u-u_0)^{+} + u_0^{+} \le u_0^{+} \le \abs{u_0},
\]
and the estimate for $u^+$ follows from that of $\abs{u_0}$.
This completes the proof.
\end{proof}

\begin{remark}
It is interesting to note that it is not clear how to deduce the previous estimate without assuming first that $\vec F=0$ and passing through the solution $\tilde u$ and Proposition~\ref{SolvabilityDelta=0}.
This follows from the fact that our proof of Lemma~\ref{ChoiceOfConstants} does not seem to work if we allow the $\vec F$ term to appear on the right hand side, without considering a special solution for the inhomogeneous terms only involving the $\vec F$ term.
\end{remark}

\begin{remark} 
The estimate in Proposition~\ref{EstimateForDelta>0} really depends on the quantity $\delta_0$.
Indeed, consider the following family functions $\set{u_s}$ for $0<s<1$:
\[
u_s(x)=\left\{\begin{array}{c l}\displaystyle (s^{1-n}-s)\,\abs{x} &\text{if }\; 0<\abs{x}<s,\\~\\\displaystyle \frac{n-1}{n-2}\,s^{2-n}-\frac{1}{2}s^2-\frac{1}{n-2}\,\abs{x}^{2-n}-\frac{1}{2}\abs{x}^2 & \text{if }\;s\leq \abs{x}<1. \end{array}\right.
\]
Then $u_s\in W_1^2(B_1)$ and and $u_s$ satisfies
\[
-\Delta u_s+d_su_s\leq n\quad\text{in }\;B_1,
\]
where
\[
d_s(x)= \frac{n-1}{\abs{x}^2}\,\chi_{(s/2,s)}(\abs{x}).
\]
Note that $\norm{d_s}_{n/2}\leq C$ and that as $s$ tends to $0$, we have $\int_{B_1}d_s \to 0$ while
\[
\int_{B_1} \abs{\nabla u_s^+}^2\geq\int_{B_s} \abs{\nabla u_s^+}^2 \to +\infty.
\]
\end{remark}

%------------------------------------------------------------------------------%
\section{Neumann Green's function}			\label{sec5}
%------------------------------------------------------------------------------%
%------------------------------------------------------------------------------%
\subsection{ Preliminary estimates}
%------------------------------------------------------------------------------%
We will construct Green's function using a duality argument.
For this, we first establish local and global pointwise estimates for solutions, in a special case.
We will use these estimates in the construction of Green's function, and we will  extend them in a more general setting later, in Section~\ref{sec6}.

\begin{proposition}			\label{prop1257}
Let $\Omega\subset\bR^n$ be a Lipschitz domain.
Let $B_r=B_r(q)$ for some $q\in\partial\Omega$ and $r<r_0$, where $r_0$ appears in Definition~\ref{LipDom}.
Let $\mathbf A=(a^{ij})$ satisfy the uniform ellipticity and boundedness condition \eqref{ellipticity}, $\vec b$, $\vec c \in L^n(\Omega)$, $d \in L^{n/2}(\Omega)$.
Assume that the pair $(\vec b, d)$ satisfies the condition \eqref{eq1122wed}.
Let $f\in L^{n/2,1}(\Omega)$ and $\vec F\in L^{n,1}(\Omega)$.
If $u\in W_1^2(\Omega)$ is a subsolution of
\[
\left\{\begin{array}{c l}-\dv(\mathbf A\nabla u+\vec b u)+\vec c\cdot \nabla u+du= f-\dv\vec F & \text{in }\;\Omega,\\
(\mathbf A \nabla u + \vec b u)\cdot \nu= \vec F\cdot\nu & \text{on }\;\partial\Omega, \end{array}\right.
\]
then there exists $C>0$, depending on $n$, $\lambda$, $\Lambda$, $\norm{\vec b-\vec c}_n$, and $M$, such that
\[
\sup_{B_r\cap\Omega}u^{+}\leq C \left(\fint_{B_{6(M+1)r}\cap\Omega}u^{+}+\norm{\vec F}_{L^{n,1}(B_{6(M+1)r}\cap\Omega)}+\norm{f^{+}}_{L^{n/2,1}(B_{6(M+1)r}\cap\Omega)}\right).
\]
\end{proposition}
\begin{proof}
Since $u$ is a subsolution, we may assume that $f\geq 0$.
Since the balls $B_{r_0}(q)$ cover $\partial\Omega$, there exists $q_i\in\partial\Omega$ as in Definition~\ref{LipDom} such that $\abs{q-q_i}<r_0$.
Consider then the Lipschitz function $\psi_i$ in Definition~\ref{LipDom}, such that
\[
B_{10(M+1)r_0}(q_i)\cap\Omega=B_{10(M+1)r_0}(q_i)\cap\set{(x',x_n)\in\bR^n: x'\in\bR^{n-1},\,x_n>\psi_i(x')}.
\]
Note that
\[
B_{2r_0}(q)\subset B_{3r_0}(q_i)\subset\Omega_{3r_0}(q_i; \psi_i) \subset B_{10(M+1)r_0}(q_i),
\]
and, consider the extensions $\tilde{\mathbf A}$, $\tilde{\vec b}$, $\tilde{\vec c}$, $\tilde{d}$, $\tilde{\vec F}$, $\tilde{f}$ and $\tilde{u}$ in $\Omega_{3r_0}=\Omega_{3r_0}(q_i; \psi_i)$ as in Lemma \ref{Extension}.
By the same lemma, we see that $\tilde{u}\in W_1^2(\Omega_{3r_0})$ is a subsolution to the equation
\[
-\dv(\tilde{\mathbf A}\nabla\tilde{u}+\tilde{\vec b}\tilde{u})+\tilde{\vec c}\nabla\tilde{u}+\tilde{d}\tilde{u}= \tilde{f}-\dv \tilde{\vec F} \quad\text{in }\;\Omega_{3r_0},
\]
and thus $\tilde{u}^{+}$ is a subsolution to the same equation in $\Omega_{3r_0}$. See \cite[Theorem 3.5]{Stamp65}. 

Moreover, by Lemma \ref{divForExtension}, we have $\tilde{d}\geq\dv\tilde{\vec b}$ in $\Omega_{3r_0}$.
Then, since $B_{2r}= B_{2r}(q)\subset B_{2r_0}(q)\subset \Omega_{3r_0}$, by \cite[Proposition 3.4]{Sak21} (see also \cite{NU12}), we have
\[
\sup_{B_r}\tilde{u}^{+}\leq C\fint_{B_{2r}}\tilde{u}^{+}+C \norm{\tilde{\vec F}}_{L^{n,1}(B_{2r})}+C\norm{\tilde f}_{L^{n/2,1}(B_{2r})},
\]
where $C$ depends on $n$, $\lambda$, $\Lambda$, and $\norm{\tilde{\vec b}-\tilde{\vec c}}_n$. 
The proof is complete by using Lemma~\ref{ReflectionNorm}.
%Hence, by the definition of $\tilde u$, we have
%\begin{equation}			\label{eq1241sat}
%\sup_{B_r}\tilde{u}^{+}\leq\frac{C}{\abs{B_{2r}}} \int_{B_{2r}\cap\Omega^{+}_{3r_0}} u^{+}+ \frac{C}{\abs{B_{2r}}} \int_{B_{2r}\cap\Omega^{-}_{3r_0}}(u')^{+},
%\end{equation}
%\todo{here, continue adding the other terms... (Yorgos)}where $(u')^{+}(y)= u^{+}(\vec \Psi^{-1}(y))$, where $\vec \Psi$ is as in \eqref{eq:Psi}.
%By a change of variable $y=\vec \Psi(x)$, we have (recall that $\abs{\det D\vec\Psi}=1$ and $\vec \Psi^{-1}=\vec \Psi$)
%\[
%\int_{B_{2r}\cap\Omega^{-}_{3r_0}}(u')^{+}(y)\,dy=\int_{\vec\Psi (B_{2r}(q)\cap\Omega^{-})}u^{+}(x)\,dx.
%\]
%Note that if $x=(x',x_n)\in B_{2r}\cap\Omega^{-}_{3r_0}$, then by writing $q=(q',\psi_i(q'))$, we find that
%\begin{align*}
%\abs{\vec\Psi(x)-q} &\leq \abs{\vec \Psi(x)-x}+\abs{x-q}=\abs{2\psi(x')-2x_n}+\abs{x-q}\\
%&\leq 2\abs{\psi(x')-\psi_i(q')}+2\abs{\psi_i(q')-x_n}+\abs{x-q} \leq 2M\abs{x'-q'}+2\abs{x-q}+\abs{x-q}.
%\end{align*}
%Hence, we have $\vec\Psi(B_{2r}\cap\Omega^{-}_{3r_0})\subset B_{6(M+1)r}\cap\Omega$.
%Therefore, the right hand side of \eqref{eq1241sat} is bounded by
%\[
%C \fint_{B_{6(M+1)r}\cap\Omega} u^{+}
%\]
%for some constant $C$ depending on $n$, $\lambda$, $\Lambda$, $\norm{\vec b-\vec c}_n$, and $M$.
\end{proof}

\begin{corollary}				\label{lem11.15}
Let $\Omega$ be a bounded Lipschitz domain.
Let $\mathbf A$ satisfy the uniform ellipticity and boundedness condition \eqref{ellipticity}, $\vec b$, $\vec c \in L^n(\Omega)$, and $d \in L^{n/2}(\Omega)$, with $(\vec b, d)$ satisfying the condition \eqref{eq1122wed} and $\abs{\Omega}^{\frac{2}{n}-1}\int_\Omega d \geq \delta_0>0$.
Let $u \in W_1^2(\Omega)$ be a solution of
\[
\left\{\begin{array}{c l}
-\dv(\mathbf A \nabla u+\vec b u)+\vec c \cdot \nabla u+d u= f& \text{in }\;\Omega,\\
(\mathbf A \nabla u+\vec b u)\cdot \vec \nu= 0 & \text{on }\;\partial\Omega,
\end{array}\right.
\]
where $f$ is a function from the Lorentz space $L^{n/2,1}(\Omega)$. 
Then we have
\[
\norm{u}_{L^\infty(\Omega)} \le C\norm{f}_{L^{n/2,1}(\Omega)},
\]
where $C$ depends on $n$, $\lambda$, $\Lambda$, $\norm{\vec b}_{n}$, $\norm{\vec c}_{n}$, $\norm{d}_{n/2}$, $\delta_0$, and the Lipschitz character of $\Omega$. 
\end{corollary}
\begin{proof}
The estimate is scale-invariant, so we may assume that $\abs{\Omega}=1$.
Let $r_0>0$ be as in Definition~\ref{LipDom}.
By the same reflection argument we used in the proof of Proposition~\ref{prop1257}, we derive from \cite[Proposition 4.6]{Sak21b} that
\begin{equation}			\label{eq1256mon}
\norm{u}_{L^\infty(\Omega \cap B_{r_0}(q))} \le C \left( \fint_{\Omega \cap B_{2r_0}(q)} \abs{u}+ \norm{f}_{L^{n/2,1}(\Omega \cap B_{2r_0}(q))} \right),
\end{equation}
for any $q \in \partial\Omega$.
Also, by Proposition~\ref{EstimateForDelta>0}, we have the global estimate
\begin{equation}			\label{eq2040sun}
\norm{u}_{L^{2n/(n-2)}(\Omega)} \le C \norm{f}_{L^{2n/(n+2)}(\Omega)},
\end{equation}
where $C$ depends on $n$, $\lambda$, $\norm{\vec b}_{L^n(\Omega)}$, $\norm{\vec c}_{L^n(\Omega)}$, $\norm{d}_{L^{n/2}(\Omega)}$, and $\delta_0$.
By H\"older's inequality and \eqref{eq2040sun}, we have
\[
\fint_{\Omega \cap B_{2r_0}(q)} \abs{u} \le C \abs{\Omega \cap B_{2r_0}(q)}^{\frac{2-n}{2n}} \norm{f}_{L^{2n/(n+2)}(\Omega)} \le C \norm{f}_{L^{2n/(n+2)}(\Omega)},
\]
where we have used Remark~\ref{r_0Bounds}.
Also, by properties of Lorentz quasi-norms (see \cite[\S~1.4.2]{Grafakos}) and, we have (recall $\abs{\Omega}=1$)
\begin{equation}			\label{eq2226sat}
\norm{f}_{L^{2n/(n+2)}(\Omega)} \le C \norm{f}_{L^{2n/(n+2),1}(\Omega)} \le C \norm{f}_{L^{n/2,1}(\Omega)}.
\end{equation}
By combining the above two estimates, we derive from \eqref{eq1256mon} that
\[
\norm{u}_{L^\infty(\Omega \cap B_{r_0}(q))} \le C\norm{f}_{L^{n/2,1}(\Omega)}.
\]
Then, the result follows from the maximum principle; see \cite[Proposition 3.4]{Sak21b}.
\end{proof}

\begin{corollary}				\label{lem11.14}
Let $\Omega$ be a bounded Lipschitz domain with $\abs{\Omega}=1$.
Let $\mathbf A$ satisfy the uniform ellipticity and boundedness condition \eqref{ellipticity} and $\vec b \in L^n(\Omega)$.
Let  $f \in L^{n/2,1}(\Omega)$ satisfying $\int_\Omega f \widehat{u}=0$, where $\widehat{u}$ is as in Proposition~\ref{ExistenceDelta=0}.
Let $u \in W_1^2(\Omega)$ be a solution of
\[
\left\{\begin{array}{c l}
-\dv(\mathbf A \nabla u)+\vec b \cdot \nabla u= f& \text{in }\;\Omega,\\
\mathbf A \nabla u\cdot \vec \nu= 0 & \text{on }\;\partial\Omega,
\end{array}\right.
\]
Then we have
\[
\norm{u}_{L^\infty(\Omega)} \le C\norm{f}_{L^{n/2,1}(\Omega)},
\]
where $C$ depends on $n$, $\lambda$, $\Lambda$, $\norm{\vec b}_{n}$, and the Lipschitz character of $\Omega$. 
\end{corollary}
\begin{proof}
The proof for Corollary~\ref{lem11.15} also works here noting that we still have \eqref{eq2040sun} by Proposition~\ref{prop2132}.
\end{proof}

We also obtain the following Caccioppoli type estimate.
	
\begin{corollary}\label{Caccioppoli}
Let $\Omega\subset \bR^n$ be a Lipschitz domain with $\abs{\Omega}=1$.
Let $\mathbf A=(a^{ij})$ satisfy the uniform ellipticity and boundedness condition \eqref{ellipticity}, $\vec b$, $\vec c \in L^n(\Omega)$, $d \in L^{n/2}(\Omega)$.
Assume that either $(\vec b, d)$ satisfies \eqref{eq1122wed} or $(\vec c,d)$ satisfies \eqref{eq1123wed}. If $u\in W_1^2(\Omega)$ is a solution to the problem
\begin{equation*}
\left\{\begin{array}{c l}
-\dv(\mathbf A \nabla u+\vec bu)+\vec c \cdot \nabla u+du = f-\dv \vec F &\text{ in }\;\Omega,\\
(\mathbf A \nabla u + \vec b u)\cdot \nu = \vec F \cdot \nu & \text{ on }\;\partial\Omega,
\end{array}\right.
\end{equation*}
then for any ball $B_{2r}$, with $r\leq\diam(\Omega)$, we have
\begin{equation*}
\int_{\Omega\cap B_r} \Abs{\nabla u}^2\le \frac{C}{r^2}\int_{\Omega\cap B_{2r}} \abs{u}^2+ C\norm{\vec F}_{L^2(\Omega\cap B_{2r})}^2+C\norm{f}_{L^{\frac{2n}{n+2}}(\Omega\cap B_{2r})}^2,
\end{equation*}
where the constant $C$ depends on $n$, $\lambda$, $\norm{\vec b-\vec c}_n$, and the Lipschitz character of $\Omega$.
\end{corollary}
\begin{proof}
The proof follows by the Caccioppoli estimate (see \cite[Theorems 3.1 and 3.2]{Mou19}), distinguishing between the cases $r<r_0$ and $r>r_0$, where $r_0$ is the constant in Definition~\ref{LipDom}, and using Remark~\ref{r_0Bounds}.
	
%Let $r_0$ be the constant in Definition~\ref{LipDom}. We distinguish between two cases: $r<\frac{r_0}{2M+5}$ and $r\geq\frac{r_0}{2M+5}$.
		
%In the first case, we distinguish between the cases $B_{2r}\cap\partial\Omega=\varnothing$ and $B_{2r}\cap\partial\Omega\neq\varnothing$. In the first case, the estimate follows by the Caccioppoli estimate (see \cite[Theorems 3.1 and 3.2]{Mou19}). In the second case, let $q\in\partial\Omega\cap B_{2r}$, then from Definition~\ref{LipDom}, there exists $q_i$ for some $i=1,\dots N$ such that $|q_i-q|<r_0$. So, in the coordinate system for $q_i$, if $\psi_i$ is as in Definition~\ref{LipDom}, we can write
%\[
%q=(q',\psi_i(q')).
%\]
%Also, for $y\in\Omega\cap B_r=B_r(x)$,
%\[
%|y-q_i|\leq|y-x|+|x-q|+|q-q_i|<r+2r+r_0<10(M+1)r_0,
%\]
%so $y=(y',s)$ in the coordinate system for $q_i$. If we not let $x_{\partial}=(x',\psi_i(x'))$, we have that $x_{\partial}\in\partial\Omega$, and
%\[
%|x_{\partial}-x|\leq|x_{\partial}-q|+|q-x|\leq|x'-q'|+|\psi_i(x')-\psi_i(q')|+|q-x|\leq 2|q-x|+M|q-x|\leq (2M+4)r.
%\]
%Hence $B_r(x)\subseteq B_{(2M+5)r}(x_{\partial})$. Therefore, in the special Lipschitz domain...
\end{proof}

%------------------------------------------------------------------------------%
\subsection{The case when$\boldsymbol{\int_\Omega d>0}$}
%------------------------------------------------------------------------------%
To construct the Neumann Green's function, we proceed similar to \cite{KS19, HK07}.
For $y \in \Omega$ and $\varepsilon>0$, set
\[
\varphi_y^\varepsilon:=\frac{1}{\abs{\Omega\cap B_\varepsilon(y)}} \chi_{\Omega\cap B_\varepsilon(y)}.
\]
By Proposition~\ref{ExistenceDelta>0}, there exists a unique solution $u_\varepsilon \in W_1^2(\Omega)$ of the problem
\begin{equation}			\label{eq1540sun}
\left\{\begin{array}{c l}
-\dv(\mathbf A \nabla u_\varepsilon+\vec b u_\varepsilon)+\vec c \cdot \nabla u_\varepsilon+d u_\varepsilon=\varphi_y^\varepsilon & \text{in }\;\Omega,\\
(\mathbf A \nabla u_\varepsilon +\vec b u_\varepsilon)\cdot \vec \nu= 0 & \text{on }\;\partial\Omega.
\end{array}\right.
\end{equation}
We set $G_\varepsilon(\cdot,y):= u_\varepsilon$.
Then, by Proposition~\ref{EstimateForDelta>0}, we have the estimate
\begin{equation*}			%\label{eq1543sun}
\norm{G_\varepsilon(\cdot, y)}_{L^{2n/(n-2)}(\Omega)} + \norm{\nabla G_\varepsilon(\cdot, y)}_{L^2(\Omega)} \le C\, \abs{\Omega\cap B_\epsilon(y)}^{\frac{2-n}{2n}}.
\end{equation*}

For a function $f \in L^{n/2,1}(\Omega)$, let us consider the problem
\[
\left\{\begin{array}{c l}
-\dv(\mathbf A\tran\nabla v+\vec c v)+\vec b \cdot \nabla v+d v=f & \text{in }\;\Omega,\\
(\mathbf A\tran \nabla v  +\vec c v) \cdot \vec \nu= 0 & \text{on }\;\partial\Omega.
\end{array}\right.
\]
Assume that $(\vec c, d)$ satisfies the condition \eqref{eq1123wed}.
Then, by  Corollary~\ref{lem11.15}, we have
\begin{equation}			\label{eq1544sun}
\norm{v}_{L^\infty(\Omega)} \le C \norm{f}_{L^{n/2,1}(\Omega)}.
\end{equation}
On the other hand, by using $v$ as a test function in \eqref{eq1540sun}, we find in light of \eqref{eq2133sat} that
\begin{equation}			\label{eq2109mon}
\int_\Omega \varphi_y^\varepsilon v 
=\int_\Omega G_\varepsilon(\cdot, y) f,
\end{equation}
and thus, by \eqref{eq1544sun} and the definition of $\varphi_y^\varepsilon$, we find that
\[
\Abs{\int_{\Omega \cap B_r(y)}  G_\varepsilon(\cdot,y) f\,} \le C \norm{f}_{L^{n/2,1}(\Omega)}.
\]
Therefore, by using the fact that $(L^{n/2,1}(\Omega))^*=L^{\frac{n}{n-2},\infty}(\Omega)$ (see \cite[\S~1.4.3] {Grafakos}), we have
\begin{equation}			\label{eq1549sun}
\norm{G_\varepsilon(\cdot, y)}_{L^{\frac{n}{n-2},\infty}(\Omega)} \le C.
\end{equation}
The Banach-Alaoglu theorem gives that there is a sequence $\set{\varepsilon_j}$ converging to zero and a function $G(\cdot, y) \in L^{\frac{n}{n-2},\infty}(\Omega)$ so that  $G_{\varepsilon_j}(\cdot,y)$ converges to $G(\cdot, y)$ in the weak-* topology of $L^{\frac{n}{n-2},\infty}(\Omega)$ and $G(\cdot, y)$ satisfies the same estimate \eqref{eq1549sun}.
By the Caccioppoli inequality (see \cite[Lemma 3.14]{KS19}), we have
\[
\int_{\Omega \setminus B_r(y)} \abs{\nabla_x G_\varepsilon(x,y)}^2 \,dx \le C(r),\quad \forall \varepsilon < r/2.
\]
This estimate will hold for the limit and thus we see that $\nabla G(\cdot, y) \in L^2(\Omega \setminus B_r(y))$ for all $r>0$.
Then, by using Corollary~\ref{Caccioppoli} and proceeding as in \cite[Proposition 3.22]{KS19}, we find that $\nabla G(\cdot, y) \in L^{\frac{n}{n-1},\infty}(\Omega)$.
Although it appears that the estimate for $\norm{\nabla G(\cdot, y)}_{L^{\frac{n}{n-1},\infty}}$ in \cite{KS19} depends on some quantity that is different from $\norm{\vec b-\vec c}_n$, the estimates indeed depend on $\norm{\vec b-\vec c}_n$ via Corollary~\ref{Caccioppoli}.

This function $G(x,y)$ is the Green's function for the problem
\begin{equation}			\label{eq1442sun}
\left\{\begin{array}{c l}
-\dv(\mathbf A\tran\nabla v+\vec c v)+\vec b \cdot \nabla v+d v=f - \dv \vec F & \text{in }\;\Omega,\\
(\mathbf A\tran \nabla v  +\vec c v) \cdot \vec \nu= g + \vec F \cdot \nu & \text{on }\;\partial\Omega.
\end{array}\right.
\end{equation}
Indeed, let $v \in W_1^2(\Omega)$ is the weak solution of the problem \eqref{eq1442sun} for Lipschitz data $f$, $\vec F$, and $g$.
Then, similar to \eqref{eq2109mon}, we have
\[
\int_\Omega \varphi_y^\varepsilon v 
=\int_\Omega G_\varepsilon(\cdot, y) f +\int_\Omega \nabla G_{\varepsilon}(\cdot,y) \cdot \vec F + \int_{\partial\Omega} G_\varepsilon(\cdot, y) g \,dS.
\]
We note that Lemma~\ref{intpl_lemma} implies that $G(\cdot, y) \in L^{\frac{n-1}{n-2},\infty}(\partial\Omega)$.
Therefore, by taking the limit $\varepsilon\to 0$ in the above, we obtain
\begin{equation}			\label{eq1539sun}
v(y)=\int_\Omega G(x,y) f(x)\,dx +\int_\Omega \nabla_x G(x,y) \cdot \vec F(x)\,dx+ \int_{\partial\Omega} G(x, y) g(x) \,dS(x),
\end{equation}
where the equality should be understood as almost everywhere sense in $y \in \Omega$.
Note that \eqref{eq1539sun} remains valid for data $f \in L^{n/2,1}(\Omega)$, $\vec F \in L^{n,1}(\Omega)$, and $g \in L^{n-1,1}(\partial\Omega)$.

\begin{definition}
We shall say that $G(x,y)$ is the Green's function for the problem \eqref{eq1442sun} if  whenever $v \in W_1^2(\Omega)$ is the weak solution of the problem \eqref{eq1442sun} for Lipschitz data $f$, $\vec F$, and $g$, then the identity \eqref{eq1539sun} holds for a.e. $y \in \Omega$.
\end{definition}
We have proved the following theorem.
\begin{theorem}			\label{thm_green1}
Let $\Omega$ be a bounded Lipschitz domain.
Let $\mathbf A$ satisfy the uniform ellipticity and boundedness condition \eqref{ellipticity}, $\vec b$, $\vec c \in L^n(\Omega)$, and $d \in L^{n/2}(\Omega)$, with $(\vec c, d)$ satisfying the condition \eqref{eq1123wed} and $\abs{\Omega}^{\frac{2}{n}-1}\int_\Omega d \geq \delta_0>0$.
Then, there exists the Green's function $G(x,y)$ for the problem \eqref{eq1442sun} and the Green's function satisfies the following:
\begin{gather*}
\nabla G(\cdot, y) \in L^2(\Omega \setminus B_r(y))\quad\text{for }\,r>0,\\
\norm{G(\cdot, y)}_{L^{\frac{n}{n-2},\infty}(\Omega)} +\norm{\nabla G(\cdot, y)}_{L^{\frac{n}{n-1},\infty}(\Omega)} + \norm{G(\cdot, y)}_{L^{\frac{n-1}{n-2},\infty}(\partial\Omega)} \le C,
\end{gather*}
where $C$ depends on $n$, $\lambda$, $\Lambda$, $\norm{\vec b}_n$, $\norm{\vec c}_n$, $\norm{d}_{n/2}$, $\delta_0$, and the Lipschitz character of $\Omega$.
Moreover, the representation formula \eqref{eq1539sun} holds if $v \in W_1^2(\Omega)$ is the weak solution of the problem \eqref{eq1442sun} with $f \in L^{n/2,1}(\Omega)$, $\vec F \in L^{n,1}(\Omega)$, and $g \in L^{n-1,1}(\partial\Omega)$.
\end{theorem}

By replacing $\mathbf A$ by $\mathbf A\tran$ and exchanging the role of $\vec b$ and $\vec c$ in the above theorem, we obtain the following corollary.
\begin{corollary}		\label{thm_green2}
Let $\Omega$ be a bounded Lipschitz domain.
Let $\mathbf A$ satisfy the uniform ellipticity and boundedness condition \eqref{ellipticity}, $\vec b$, $\vec c \in L^n(\Omega)$, and $d \in L^{n/2}(\Omega)$, with $(\vec b, d)$ satisfying the condition \eqref{eq1122wed} and $\abs{\Omega}^{\frac{2}{n}-1}\int_\Omega d \geq \delta_0>0$.
Then, there exists the Green's function $G^*(x,y)$ for the problem \begin{equation}			\label{eq1726sun}
\left\{\begin{array}{c l}
-\dv(\mathbf A \nabla v+\vec b v)+\vec c \cdot \nabla v+d v=f - \dv \vec F & \text{in }\;\Omega,\\
(\mathbf A \nabla v  +\vec c v) \cdot \vec \nu= g + \vec F \cdot \nu & \text{on }\;\partial\Omega,
\end{array}\right.
\end{equation}
which satisfies the following:
\begin{gather*}
\nabla G^*(\cdot, y) \in L^2(\Omega \setminus B_r(y))\quad\text{for }\,r>0,\\\norm{G^*(\cdot, y)}_{L^{\frac{n}{n-2},\infty}(\Omega)} +\norm{\nabla G^*(\cdot, y)}_{L^{\frac{n}{n-1},\infty}(\Omega)} + \norm{G^*(\cdot, y)}_{L^{\frac{n-1}{n-2},\infty}(\partial\Omega)} \le C,
\end{gather*}
where $C$ depends on $n$, $\lambda$, $\Lambda$, $\norm{\vec b}_n$, $\norm{\vec c}_n$, $\norm{d}_{n/2}$, $\delta_0$, and the Lipschitz character of $\Omega$.
Moreover, if $v \in W_1^2(\Omega)$ is the weak solution of the problem \eqref{eq1726sun} with $f \in L^{n/2,1}(\Omega)$, $\vec F \in L^{n,1}(\Omega)$, and $g \in L^{n-1,1}(\partial\Omega)$, then we have
\begin{equation}			\label{eq2334sun}
v(y)=\int_\Omega G^*(x,y) f(x)\,dx+ \int_\Omega \nabla_x G^*(x,y) \cdot \vec F(x)\,dx+\int_{\partial\Omega} G^*(x,y) g(x) \,dS(x).
\end{equation}
\end{corollary}

The above two results are analogues of \cite[Theorem~7.2]{KS19}.
If $(\vec b, d)$ satisfies \eqref{eq1122wed} and $(\vec c,d)$ satisfies \eqref{eq1123wed} simultaneously, then similar to \cite[Theorem~7.9]{KS19}, we can derive the pointwise bound for the Green's function.
To see this, observe that for any $r<\dist(y,\partial\Omega)$, the function $u=G(\cdot, y)$ satisfies
\[
\left\{\begin{array}{c l}
-\dv(\mathbf A \nabla u+\vec b u)+\vec c \cdot \nabla u+d u=0 & \text{in }\;\Omega\setminus B_r(y),\\
(\mathbf A \nabla u +\vec b u)\cdot \vec \nu= 0 & \text{on }\;\partial\Omega.
\end{array}\right.
\]
Suppose that $\abs{\Omega}=1$ for now.
For $x \neq y$, set $r= \frac14 \abs{x-y}$ and take $r_0$ from Definition~\ref{LipDom}.
We distinguish two cases: $r<r_0$ or $r\geq r_0$.
In the case when $r<r_0$ and $B_{2r}(x) \subset \Omega$, we use the interior estimate to get
\begin{equation}			\label{eq0958mon}
	\abs{G(x,y)} \le \frac{C}{r^n} \int_{B_{2r}(x)} \abs{G(\cdot, y)} \le \frac{C}{r^n} r^2 \norm{G(\cdot, y)}_{L^{\frac{n}{n-2},\infty}(\Omega)} \le C \abs{x-y}^{2-n}.
\end{equation}
In the case when $r<r_0$ and $B_{2r}(x) \not\subset \Omega$, we use an analogue of Proposition~\ref{prop1257} instead of the interior estimate to get the same bound \eqref{eq0958mon}.

In the case when $r \ge r_0$, we further consider two cases:  $\dist(x,\partial\Omega)<r_0$ or $\dist(x,\partial\Omega) \ge r_0$.
If $\dist(x,\partial\Omega)<r_0$, we use an analogue of Proposition~\ref{prop1257}, and otherwise, we use the interior estimate, respectively, to obtain
\[
\abs{G(x,y)}\leq Cr_0^{2-n}.
\]
Then, we have
\[
c_0 \le r_0 \le r<\diam(\Omega) \le c_1,
\]
where $c_0$ and $c_1$ are positive numbers that depend only on $n$ and the Lipschitz character of $\Omega$ (see Remark~\ref{r_0Bounds}), hence \eqref{eq0958mon} still holds in this case.

We note that in Corollary~\ref{thm_green2}, the construction of Green's function $G^*(x,y)$ yields the following identity for $y \neq x$:
\[
G^*(y,x)=G(x,y).
\]
Moreover, by construction of Green's function, if $G_{\Omega}$ is Green's function for $\Omega$, $x,y\in\Omega$ and $\Omega_r=\frac{1}{r}\Omega$, then we have
\begin{equation}\label{eq:GreenScaling}
G_{\Omega}(x,y)=r^{2-n}G_{\Omega_r}(x/r,y/r).
\end{equation}
Hence, we have the following theorem.
\begin{theorem}			%\label{thm_green3}
Let $\Omega$ be a bounded Lipschitz domain.
Let $\mathbf A$ satisfy the condition \eqref{ellipticity}, $\vec b$, $\vec c \in L^n(\Omega)$, $d \in L^{n/2}(\Omega)$, and $\abs{\Omega}^{\frac{2}{n}-1}\int_\Omega d \geq \delta_0>0$.
Assume that $(\vec b, d)$ satisfies \eqref{eq1122wed} and $(\vec c,d)$ satisfies \eqref{eq1123wed} simultaneously.
Then the conclusions of Theorem~\ref{thm_green1} and Corollary~\ref{thm_green2} hold and  we have
\[
G(x,y)=G^*(y,x),
\]
where the equality should be understood in the almost everywhere sense for $(x,y) \in \Omega\times \Omega$.
Moreover, we have the pointwise bound
\[
\abs{G(x,y)} \le C \abs{x-y}^{2-n}\quad\text{for }x\neq y,
\]
where $C$ depends on $n$, $\lambda$, $\Lambda$, $\norm{\vec b}_n$, $\norm{\vec c}_n$, $\norm{d}_{n/2}$, $\delta_0$, and the Lipschitz character of $\Omega$.
\end{theorem}
\begin{proof}
Using \eqref{eq:GreenScaling}, we see that the estimate is scale invariant, so we may assume that $\abs{\Omega}=1$. Then, the proof follows from the previous discussion.	
\end{proof}
%------------------------------------------------------------------------------%
\subsection{The case when$\boldsymbol{\int_\Omega d=0}$}
%------------------------------------------------------------------------------%
Now we consider the case when $\int_\Omega d=0$.
Instead of \eqref{eq1540sun}, we consider the problem
\begin{equation}			\label{eq2010sun}
\left\{\begin{array}{c l}
-\dv(\mathbf A \nabla u_\varepsilon+(\vec b -\vec c) u_\varepsilon)=\varphi_\varepsilon - \frac{1}{\abs{\Omega}} & \text{in}\,\,\,\Omega,\\
(\mathbf A \nabla u_\varepsilon -(\vec b-\vec c)u_\varepsilon)\cdot \vec \nu = 0 & \text{on}\,\,\,\partial\Omega.
\end{array}\right.
\end{equation}
Since $\int_\Omega \left( \varphi_\varepsilon -\frac{1}{\abs{\Omega}} \right)=0$, by Proposition~\ref{ExistenceDelta=0}, there exist a unique solution  to \eqref{eq2010sun} satisfying $\int_\Omega u_\varepsilon=0$.
As before, we set $G_\varepsilon(\cdot, y)=u_\varepsilon$.

For a function $f \in L^{n/2,1}(\Omega)$, we set
\[
\overline{f}= \frac{\int_\Omega f \widehat{u}}{\int_\Omega \widehat{u}}
\]
so that $\int_\Omega (f-\overline{f})\,\widehat{u}=0$.
By Proposition~\ref{ExistenceDelta=0Adj}, there exists a solution $v$ of the problem
\[
\left\{\begin{array}{c l}
-\dv(\mathbf A\tran\nabla v)+(\vec b-\vec c) \cdot \nabla v=f-\overline{f} & \text{in}\,\,\,\Omega,\\
\mathbf A\tran \nabla v \cdot \vec \nu = 0 & \text{on}\,\,\,\partial\Omega,
\end{array}\right.
\]
satisfying  satisfying $\int_\Omega v=0$ .
Then, we have
\begin{equation}		\label{eq2207sat}
\int_\Omega \varphi_\varepsilon v =\int_\Omega \left( \varphi_\varepsilon -\frac{1}{\abs{\Omega}} \right) v =\int_\Omega G_\varepsilon(\cdot, y) (f-\overline{f})=\int_\Omega G_\varepsilon(\cdot, y) f,
\end{equation}
which agrees with \eqref{eq2109mon}.
On the other hand, by Corollary~\ref{lem11.14}, we find
\[
\norm{v}_{L^\infty(\Omega)} \le C \norm{f-\overline{f}}_{L^{n/2,1}(\Omega)}.
\]
Therefore, by \eqref{eq2207sat} and the definition of $\varphi_\varepsilon$, we find that
\begin{equation}			\label{eq2232sat}
\Abs{\int_\Omega G_\varepsilon(\cdot, y) f} \le C \norm{f-\overline{f}}_{L^{n/2,1}(\Omega)}.
\end{equation}
In order to estimate $\abs{\overline{f}}$, first recall that $\norm{\widehat{u}}_{L^{\frac{2n}{n-2}}(\Omega)}=1$ and use \eqref{eq2226sat} to derive
\[
\Abs{\int_\Omega f \widehat{u}\,} \le \norm{f}_{L^{\frac{2n}{n+2}}(\Omega)} \le C \abs{\Omega}^{\frac{n-2}{2n}}\,\norm{f}_{L^{n/2,1}(\Omega)}.
\]
Next, we use Proposition~\ref{EstimateForDelta=0} and the fact $\widehat{u}>0$ to obtain
\[
1=\norm{\widehat{u}}_{L^{\frac{2n}{n-2}}(\Omega)} \le \norm{\widehat{u}}_{Y_1^2(\Omega)} \le C \abs{\Omega}^{-\frac{n+2}{2n}} \int_\Omega \widehat{u}.
\]
By combining the previous two inequalities, we obtain
\[
\abs{\overline{f}} \le C \abs{\Omega}^{-\frac{2}{n}} \norm{f}_{L^{n/2,1}(\Omega)}.
\]
Therefore, by the inequalities for Lorentz norms (see \cite{Grafakos})
\[
\norm{f-\overline{f}}_{L^{n/2,1}(\Omega)} \le C \norm{f}_{L^{n/2,1}(\Omega)} + C \norm{\overline{f}}_{L^{n/2,1}(\Omega)}  \le C \norm{f}_{L^{n/2,1}(\Omega)}.
\]
This combined with \eqref{eq2232sat} and duals of Lorentz spaces, we have
\[
\norm{G_\varepsilon(\cdot, y)}_{L^{\frac{n}{n-2},\infty}(\Omega)} \le C,
\]
which coincides with \eqref{eq1549sun}.
Then by replicating the same arguments, we obtain the following results.

\begin{theorem}			\label{thm_green1b}
Let $\Omega$ be a bounded Lipschitz domain.
Let $\mathbf A$ satisfy the uniform ellipticity and boundedness condition \eqref{ellipticity}, $\vec b$, $\vec c \in L^n(\Omega)$, and $d \in L^{n/2}(\Omega)$, with $(\vec c, d)$ satisfying the condition \eqref{eq1123wed} and $\int_\Omega d =0$.
Then, there exists the Green's function $G(x,y)$ for the problem \eqref{eq1442sun} and the Green's function satisfies the following:
\begin{gather*}
\nabla G(\cdot, y) \in L^2(\Omega \setminus B_r(y))\quad\text{for }\,r>0,\\
\norm{G(\cdot, y)}_{L^{\frac{n}{n-2},\infty}(\Omega)} +\norm{\nabla G(\cdot, y)}_{L^{\frac{n}{n-1},\infty}(\Omega)} + \norm{G(\cdot, y)}_{L^{\frac{n-1}{n-2},\infty}(\partial\Omega)} \le C,
\end{gather*}
where $C$ depends on $n$, $\lambda$, $\Lambda$, $\norm{\vec b-\vec c}_n$, $\norm{d}_{n/2}$, and the Lipschitz character of $\Omega$.
Moreover, the representation formula \eqref{eq1539sun} holds if $v \in W_1^2(\Omega)$ is the weak solution of the problem \eqref{eq1442sun} with $f \in L^{n/2,1}(\Omega)$, $\vec F \in L^{n,1}(\Omega)$, and $g \in L^{n-1,1}(\partial\Omega)$.
\end{theorem}

\begin{corollary}		\label{thm_green2b}
Let $\Omega$ be a bounded Lipschitz domain.
Let $\mathbf A$ satisfy the uniform ellipticity and boundedness condition \eqref{ellipticity}, $\vec b$, $\vec c \in L^n(\Omega)$, and $d \in L^{n/2}(\Omega)$, with $(\vec b, d)$ satisfying the condition \eqref{eq1122wed} and $\int_\Omega d =0$.
Then, there exists the Green's function $G^*(x,y)$ for the problem \eqref{eq1726sun} and it satisfies the following:
\begin{gather*}
\nabla G^*(\cdot, y) \in L^2(\Omega \setminus B_r(y))\quad\text{for }\,r>0,\\\norm{G^*(\cdot, y)}_{L^{\frac{n}{n-2},\infty}(\Omega)} +\norm{\nabla G^*(\cdot, y)}_{L^{\frac{n}{n-1},\infty}(\Omega)} + \norm{G^*(\cdot, y)}_{L^{\frac{n-1}{n-2},\infty}(\partial\Omega)} \le C,
\end{gather*}
where $C$ depends on $n$, $\lambda$, $\Lambda$, $\norm{\vec b-\vec c}_n$, $\norm{d}_{n/2}$, and the Lipschitz character of $\Omega$.
Moreover, if $v \in W_1^2(\Omega)$ is the weak solution of the problem \eqref{eq1726sun} with $f \in L^{n/2,1}(\Omega)$, $\vec F \in L^{n,1}(\Omega)$, and $g \in L^{n-1,1}(\partial\Omega)$,  then the formula \eqref{eq2334sun} holds.
\end{corollary}

\begin{theorem}			%\label{thm_green3b}
Let $\Omega$ be a bounded Lipschitz domain.
Let $\mathbf A$ satisfy the condition \eqref{ellipticity}, $\vec b$, $\vec c \in L^n(\Omega)$, $d \in L^{n/2}(\Omega)$, and $\int_\Omega d =0$.
Assume that $(\vec b, d)$ satisfies \eqref{eq1122wed} and $(\vec c,d)$ satisfies \eqref{eq1123wed} simultaneously.
Then the conclusions of Theorem~\ref{thm_green1b} and Corollary~\ref{thm_green2b} hold and  we have
\[
G(x,y)=G^*(y,x),
\]
where the equality should be understood in the almost everywhere sense for $(x,y) \in \Omega\times \Omega$.
Moreover, we have the pointwise bound
\[
\abs{G(x,y)} \le C \abs{x-y}^{2-n}\quad\text{for }x\neq y,
\]
where $C$ depends on $n$, $\lambda$, $\Lambda$, $\norm{\vec b-\vec c}_n$, $\norm{d}_{n/2}$, and the Lipschitz character of $\Omega$.
\end{theorem}

%------------------------------------------------------------------------------%
\section{Scale invariant boundedness estimates}	\label{sec6}
%------------------------------------------------------------------------------%

As an application of Green's function, we extend our boundedness results in the beginning of Section~\ref{sec5} to include Neumann data $g$.
Note that subsolutions to the problem
\begin{equation*}
\left\{\begin{array}{c l}
-\dv(\mathbf A \nabla u+\vec b u)+\vec c \cdot \nabla u+d v=f-\dv \vec F & \text{in }\;\Omega,\\
(\mathbf A \nabla u  +\vec b u) \cdot \vec \nu= g+\vec F\cdot\nu & \text{on }\;\partial\Omega,
\end{array}\right.
\end{equation*}
for $g \neq 0$ cannot be reduced to subsolutions of a Dirichlet problem in a larger domain, so the use of Green's function is necessary in the arguments that follow.

The first estimate is the analogue of Proposition~\ref{prop1257}, in the general setting we consider in this article.

\begin{proposition}
Let $\Omega\subset\bR^n$ be a Lipschitz domain.
Let $\mathbf A=(a^{ij})$ satisfy the uniform ellipticity and boundedness condition \eqref{ellipticity}, $\vec b$, $\vec c \in L^n(\Omega)$, $d \in L^{n/2}(\Omega)$.
Assume that the pair $(\vec b, d)$ satisfies the condition \eqref{eq1122wed}.
Suppose $f\in L^{n/2,1}(\Omega)$, $\vec F\in L^{n,1}(\Omega)$, $g\in L^{n-1,1}(\partial\Omega)$, and $u\in W_1^2(\Omega)$ is a subsolution of
\[
\left\{\begin{array}{c l}-\dv(\mathbf A\nabla u+\vec b u)+\vec c\cdot \nabla u+du= f-\dv \vec F, & \text{in }\;\Omega,\\
(\mathbf A \nabla u + \vec b u)\cdot \nu= g+ \vec F \cdot \nu & \text{on }\;\partial\Omega. \end{array}\right.
\]
Let $B_r=B_r(q)$ for some $q\in\partial\Omega$ and $r<r_0$, where $r_0$ appears in Definition~\ref{LipDom}, and denote $\Omega_r=\Omega \cap B_r$ and $\Gamma_r=\partial\Omega\cap B_r.$
Then, we have
%\[
%\sup_{B_r\cap\Omega}u^{+}
%\leq C \left (\fint_{B_{6(M+1)r}\cap\Omega}u^{+}+\norm{\vec F}_{L^{n,1}(B_{6(M+1)r}\cap\Omega)}
%+\norm{f^{+}}_{L^{n/2,1}(B_{6(M+1)r}\cap\Omega)}+\norm{g^{+}}_{L^{n-1,1}(B_{6(M+1)r}\cap\partial\Omega)} \right),
%\]
\[
\sup_{\Omega_r}u^{+}
\leq C \left (\fint_{\Omega_{6(M+1)r}}u^{+}+\norm{\vec F}_{L^{n,1}(\Omega_{6(M+1)r})}
+\norm{f^{+}}_{L^{n/2,1}(\Omega_{6(M+1)r})}+\norm{g^{+}}_{L^{n-1,1}(
\Gamma_{6(M+1)r})}\right),
\]
where $C$ is a constant depending on $n$, $\lambda$, $\Lambda$, $\norm{\vec b}_n$, $\norm{\vec c}_n$, $\norm{d}_{n/2}$, and $M$.
\end{proposition}
\begin{proof}
We may assume that $f,g\geq 0$.
Since $q\in\partial\Omega$, consider $q_i\in\partial\Omega$ and the Lipschitz function $\psi_i:\bR^{n-1}\to\bR$ from Definition~\ref{LipDom} such that $\abs{q-q_i}<r_0$.
After rotating and translating, we may assume that the domain $U=\Omega^{+}_{2r}(q_i; \psi_i)$ is a special Lipschitz domain and it is a subset of $\Omega$.
Let $v\in W_1^2(U)$ be the solution of the problem
\[
\left\{\begin{array}{c l}-\dv(\mathbf A\nabla v)+\vec c\cdot \nabla v= 0, & \text{in }\;U,\\
\mathbf A \nabla v \cdot \nu= \hat{g} & \text{on }\;\partial U, \end{array}\right.
\]
where $\hat{g}=g\chi_{\partial U \cap B_{2r}}+\overline{g}\chi_{\partial U \setminus B_{2r}}$, and the constant $\overline{g}$ is chosen to satisfy
\[
\overline{g} \,\Abs{\partial U \setminus B_{2r}}+\int_{\partial U \cap B_{2r}}g=0.
\]
The existence and uniqueness of the solution $v$ is guaranteed by Proposition~\ref{SolvabilityDelta=0}.
Then, if $G_U(x,y)$ is the Green's function for the same problem in $U$, it follows from Theorem~\ref{thm_green1b} that for $y\in \Omega$, we have
\begin{equation}		\label{eq:vpointwise}
\begin{aligned}
\abs{v(y)}&=\Abs{\int_{\partial U} G_U(x,y)\hat{g}(x)\,dS(x)}\\
&\leq \norm{G_U(\cdot,y)}_{L^{\frac{n-1}{n-2},\infty}(\partial U^{+}_{2r})} \left( \norm{g}_{L^{n-1,1}(\partial U^{+}_{2r}\cap B_{2r})}+\norm{\overline{g}}_{L^{n-1,1}(\partial U^{+}_{2r}\setminus B_{2r})}\right)\\
&\leq C\norm{g}_{L^{n-1,1}(\partial\Omega\cap B_{2r})},
\end{aligned}
\end{equation}
where $C$ depends on $n$, $\lambda$, $\Lambda$, $\norm{\vec c}_n$, and $M$ (because $U$ is a Lipschitz domain with Lipschitz character $(M,N)$ and $N$ depends only on $n$ and $M$).
If we set $w=u-v$, then $w$ becomes a subsolution of
\[
\left\{\begin{array}{c l}-\dv(\mathbf A\nabla w+\vec b w)+\vec c\cdot \nabla w+dw= (f-dv)-\dv(\vec F-\vec b v), & \text{in }\;U,\\
(\mathbf A \nabla w + \vec b w)\cdot \nu= (\vec F-\vec b v) \cdot \nu & \text{on }\;\partial U\cap B_{2r},\\ 
(\mathbf A \nabla w + \vec b w)\cdot \nu= g-\overline{g}+(\vec F-\vec b v) \cdot \nu & \text{on }\;\partial U \setminus B_{2r}
\end{array}\right.
\]
We then follow the steps of the proof of Proposition~\ref{prop1257} and use \eqref{eq:vpointwise} to complete the proof.
\end{proof}

Using the maximum principle argument in the proof of Corollary~\ref{lem11.15}, we have shown the following.

\begin{proposition}\label{pointwiseSubsol}
Let $\Omega\subset\bR^n$ be a Lipschitz domain.
Let $\mathbf A=(a^{ij})$ satisfy the uniform ellipticity and boundedness condition \eqref{ellipticity}, $\vec b$, $\vec c \in L^n(\Omega)$, $d \in L^{n/2}(\Omega)$.
Assume that the pair $(\vec b, d)$ satisfies the condition \eqref{eq1122wed}.
Suppose $f\in L^{n/2,1}(\Omega),\vec F\in L^{n,1}(\Omega)$, $g\in L^{n-1,1}(\partial\Omega)$, and $u\in W_1^2(\Omega)$ is a subsolution of
\[
\left\{\begin{array}{c l}-\dv(\mathbf A\nabla u+\vec b u)+\vec c\cdot \nabla u+du= f-\dv \vec F, & \text{in }\;\Omega,\\
(\mathbf A \nabla u + \vec b u)\cdot \nu= g+ \vec F \cdot \nu & \text{on }\;\partial\Omega. \end{array}\right.
\]
Then there exists $C>0$, depending on $n$, $\lambda$, $\Lambda$, $\norm{\vec b}_n$, $\norm{\vec c}_n$, $\norm{d}_{n/2}$ and the Lipschitz character of $\Omega$, such that
\[
\sup_{\Omega} u^+\leq C \left(\fint_{\Omega}u^{+}+\norm{\vec F}_{L^{n,1}(\Omega)}+\norm{f^{+}}_{L^{n/2,1}(\Omega)}+\norm{g^{+}}_{L^{n-1,1}(\partial\Omega)}\right).
\]
\end{proposition}

Finally, we have the following scale invariant pointwise estimates for solutions.

\begin{proposition}
Let $\Omega\subset\bR^n$ be a Lipschitz domain.
Let $\mathbf A=(a^{ij})$ satisfy the uniform ellipticity and boundedness condition \eqref{ellipticity}, $\vec b$, $\vec c \in L^n$, $d \in L^{n/2}$.
Assume that the pair $(\vec b, d)$ satisfies the condition \eqref{eq1122wed}.
Assume that  $\abs{\Omega}^{\frac{2}{n}-1}\int_\Omega d \geq \delta_0>0$.
Let $f\in L^{n/2,1}(\Omega)$, $\vec F\in L^{n,1}(\Omega)$, $g\in L^{n-1,1}(\partial\Omega)$.
If $u\in W_1^2(\Omega)$ is the solution to the problem
\[
\left\{\begin{array}{c l}
-\dv(\mathbf A \nabla u+\vec bu)+\vec c \cdot \nabla u+du = f-\dv \vec F &\text{ in }\;\Omega,\\
(\mathbf A \nabla u + \vec b u)\cdot \nu = g+ \vec F \cdot \nu & \text{ on }\;\partial\Omega,
\end{array}\right.
\]
then there exists $C>0$, depending on $n$, $\lambda$, $\norm{\vec b}_n$, $\norm{\vec c}_n$, $\norm{d}_{n/2}$, $\delta_0$, and the Lipschitz character of $\Omega$, such that
\[
\sup_{\Omega} \, \abs{u} \leq C\left(\norm{\vec F}_{L^{n,1}(\Omega)}+\norm{f}_{L^{n/2,1}(\Omega)}+\norm{g}_{L^{n-1,1}(\partial\Omega)}\right).
\]
\end{proposition}
\begin{proof}
The proof follows applying Proposition~\ref{pointwiseSubsol} to $u$ and $-u$, and combining with Proposition~\ref{EstimateForDelta>0}.
\end{proof}

\begin{proposition}
Let $\Omega\subset \bR^n$ be a bounded Lipschitz domain.
Assume that $\mathbf A=(a^{ij})$ satisfies the uniform ellipticity and boundedness condition \eqref{ellipticity}, and $\vec c \in L^n(\Omega)$.
Let $f\in L^{n/2,1}(\Omega)$, $\vec F\in L^{n,1}(\Omega)$, $g\in L^{n-1,1}(\partial\Omega)$, and assume that the compatibility condition  \eqref{eq:Comp2} holds.
Then the solution $u\in W_1^2(\Omega)$ to the problem
\[
\left\{\begin{array}{c l}-\dv(\mathbf A\nabla u)+\vec c \cdot \nabla u=f-\dv \vec F& \text{in }\;\Omega,\\ \mathbf A \nabla u\cdot \nu=g+\vec F\cdot \nu & \text{on}\,\,\,\partial\Omega,\end{array}\right.
\]
with $\int_\Omega u=0$, satisfies the estimate
\[
\sup_{\Omega} \, \abs{u} \leq C \left(\norm{\vec F}_{L^{n,1}(\Omega)}+\norm{f}_{L^{n/2,1}(\Omega)}+\norm{g}_{L^{n-1,1}(\partial\Omega)} \right).
\]
where $C$ depends on $n$, $\lambda$, $\Lambda$, $\norm{\vec c}_n$, and the Lipschitz character of $\Omega$.
\end{proposition}
\begin{proof}
Existence and uniqueness follows from Proposition~\ref{prop2132}. The pointwise estimate follows from Theorem~\ref{thm_green1b}.
\end{proof}

%------------------------------------------------------------------------------%
\section{Appendix}
%------------------------------------------------------------------------------%
\subsection{On the assumption \eqref{eq1122wed}}
%------------------------------------------------------------------------------%

It turns out that both our assumptions $d\geq\dv \vec b$ and $\vec b\cdot \nu\geq 0$ are necessary for the theorems we have treated in this article.
In particular, the absence of either of the two can lead to solutions that are unbounded close to the boundary, or spaces of solutions to the homogeneous Neumann problem with dimension strictly greater than $1$.

The assumption $d\geq\dv \vec b$ has its roots in the treatment of the Dirichlet problem (see \cite[Section 8.1]{GT01}) and it is connected to the positivity of the eigenvalues of the equation $-\Delta u=\lambda u$.
Indeed, if $\Omega=(0,\pi)^3 \subset \bR^3$, the problem
\[
\left\{\begin{array}{c l}-\Delta u-u=0 & \text{in }\;\Omega,\\
\partial u/\partial n=0 & \text{on }\;\partial\Omega,\end{array}\right.
\]
has at least three linearly independent solutions, namely, $\cos x$, $\cos y$, and $\cos z$.
Hence, an analogue of Proposition~\ref{kernel} is not possible.

Also, in the case when $n \ge 3$ and
\[
\Omega=\set{x=(x',x_n)\in \bR^n : \abs{x'}^2+x_n^2<e^{-2},\, x_n>0},
\]
then the functions
\[
u(x)=\ln \abs{x},\quad \vec b(x)=-\frac{x}{\abs{x}^2 \ln\abs{x}}
\]
satisfy $u \in W_1^2(\Omega)$, $\vec b\in L^n(\Omega)$, and $\dv \vec b>0$ in $\Omega$; see \cite[Section 7.2]{KS19}.
A direct computation shows that $\nabla u+\vec b u=0$, and thus $u$ is a solution to the Neumann problem
\[
\left\{\begin{array}{c l}-\Delta u-\dv(\vec bu)=0& \text{in }\;\Omega,\\
(\nabla u +b\vec u)\cdot \nu=0& \text{on }\;\partial\Omega,\end{array}\right.
\]
However, $u$ is not bounded near $0$.

Even under the assumption $d\geq\dv \vec b$, the further assumption $\vec b\cdot \nu\geq 0$ is still necessary for Propositions~\ref{kernel} and for boundedness.
This is demonstrated in the following examples.

\begin{example}
Consider the Lipschitz domain
\[
\Omega=\Set{x=(x',x_n)\in \bR^{n}: \abs{x'}^2+x_n^2<e^{-2}, \, x_n>\abs{x'}}
\]
and set
\[
u(x)=u(x_n)=-\ln x_n,\quad\vec b(x)=\vec b(x_n)=-\frac{\vec e_n}{x_n \ln x_n}.
\]
Since $x_n \simeq \abs{x}$ in $\Omega$, we have $u\in W_1^2(\Omega)$ and $\vec b\in L^n(\Omega)$, as long as $n\geq 3$.
Moreover, we have $\dv \vec b \leq 0$ and since $\nabla u+\vec b u=0$, $u$ is a solution to the Neumann problem
\[
\left\{\begin{array}{c l}-\Delta u-\dv(\vec bu)=0& \text{in }\;\Omega,\\
(\nabla u +b\vec u)\cdot \nu=0& \text{on }\;\partial\Omega.\end{array}\right.
\]
However, $u$ is not bounded near  $0\in\partial\Omega$.
Observe that $\vec b\cdot \nu=\frac{1}{\sqrt{2}y\ln y}<0$ near $0$.
\end{example}

To construct kernels with dimensions greater than $1$, we consider a setting where we can apply the separations of variables method, after we construct a one-dimensional counterexample.
\begin{example}
Consider $B(x)=x^2-\delta x$, where $\delta>0$ to be determined later, and set $b(x)=B'(x)=2x-\delta$.
Then, we have $b(-1)=-2-\delta<0$, and by setting
\begin{equation}\label{eq:uDfn}
u(x)=\frac{e^{B(-1)-B(x)}}{\beta(-1)}+e^{-B(x)}\int_{-1}^x e^{B(t)}\,dt,
\end{equation}
we get
\[
u'(x)=-\frac{b(x) e^{B(-1)-B(x)}}{b(-1)}-b(x)e^{-B(x)}\int_{-1}^x e^{B(t)}\,dt+1=-b(x)u(x)+1.
\]
Note that $u'(-1)=0$ and also that
\[
u'(1)=-\frac{b(1)e^{B(-1)-B(1)}}{b(-1)}-b(1)e^{-B(1)}\int_{-1}^1e^{B(t)}\,dt+1.
\]
Then, we find that $u'(1)=0$ if and only if $f(\delta)=1$, where
\[
f(x):=(2-x)e^{x-1}\left(-\frac{e^{x+1}}{x+2}+\int_{-1}^1e^{t^2-x t}\,dt\right).
\]
Since $\int_0^1 e^{t^2}\,dt > e/2$, we have $f(0)>1$, while $f(2)=0$.
It is clear that $f(x)$ is continuous for $0 \le x \le 2$, and thus there exists $\delta\in(0,2)$ such that $f(\delta)=1$.
Therefore, with this choice of $\delta$, the function $u$ in \eqref{eq:uDfn} satisfies  $u'(1)=0$ and thus it solves the one-dimensional Neumann problem
\[
-(u'+b u)'=0\;\text{ in }\;(0,1),\quad u'(-1)=u'(1)=0.
\]
Now, let us define
\[
\vec b(x,y,z)=(-b(x),-b(y),b(z)).
\]
Note that $\dv \vec b=-2$.
If we set $v(x,y,z)=u(x)$ in $\Omega=(-1,1)^3\subset\bR^3$, then we have
\[
-\Delta v +\vec b\cdot \nabla v-2v=-u''(x)-b(x)u'(x)-2u(x)=-u''(x)-(b(x)u(x))'=0
\]
and $\partial v/\partial \nu=0$ on $\partial\Omega$.
Similarly, if we set $w(x,y,z)=u(y)$, we have
\[
-\Delta w +\vec b\cdot \nabla w-2w=0\;\text{ in }\;\Omega,\quad \partial w/\partial \nu=0\;\text{ on }\;\partial \Omega.
\]
So, the solution space for the Neumann problem
\[
\left\{\begin{array}{c l}-\Delta u+\vec b\cdot \nabla u-2u=0 & \text{in }\;\Omega,\\ \partial u/\partial \nu=0 & \text{on }\;\partial\Omega,\end{array}\right.
\]
has dimension greater than $1$.
Hence, by the Fredholm alternative, the space for $W_1^2(\Omega)$ solutions to the Neumann problem
\[
\left\{\begin{array}{c l}-\Delta u-\dv(\vec bu)-2u=0 & \text{in }\;\Omega,\\
(\nabla u+\vec bu)\cdot\nu=0 & \text{on }\;\partial\Omega,\end{array}\right.
\]
has dimension greater than $1$.
So an analogue of Proposition~\ref{kernel} does not hold, if we only assume that $d\geq\dv \vec b$.
\end{example}

\subsection{A case when subsolutions are solutions}

It turns out that the definition of subsolutions is in fact strong enough to force subsolutions to be solutions, at least in some specific cases. This is the context of the following proposition.

\begin{proposition}\label{SubImplies Sol1}
Let $\Omega$ be a bounded Lipschitz domain.
Let $\mathbf A$ satisfy the uniform ellipticity and boundedness condition \eqref{ellipticity} and $\vec b \in L^n(\Omega)$.
Also, let $f\in L^{\frac{2n}{n+2}}(\Omega)$, $\vec F\in L^2(\Omega)$ and $g\in L^{2-\frac{2}{n}}(\Omega)$ satisfy the compatibility condition \eqref{eq:Comp1}.
If $u\in Y_1^2(\Omega)$ is a subsolution to the Neumann problem
\begin{equation}			\label{eq1311sun}
\left\{\begin{array}{c l}-\dv(\mathbf A\nabla u+\vec bu)=f-\dv \vec F & \text{in }\;\Omega,
\\ (\mathbf A \nabla u+\vec b u)\cdot \nu= g+\vec F\cdot\nu & \text{on }\;\partial\Omega,\end{array}\right.
\end{equation}
then $u$ is, in fact, a solution to the same problem \eqref{eq1311sun}.
\end{proposition}
\begin{proof}
Consider the solution $v$ to the problem \eqref{eq1311sun}, which exists by Proposition~\ref{ExistenceDelta=0}.
Let $\widehat{u}\in Y_1^2(\Omega)$ be as in Proposition~\ref{ExistenceDelta=0} and set $\tilde u=u-v-c\widehat{u}$, where $c\in\bR$ is chosen so that $\int_{\Omega} \tilde u=0$.
Then $u_0$ is a subsolution to the Neumann problem
\[
\left\{\begin{array}{c l}-\dv(\mathbf A\nabla \tilde  u+\vec b \tilde u)=0 & \text{in }\;\Omega,
\\ (\mathbf A \nabla \tilde u+\vec b \tilde u)\cdot \nu= 0 & \text{on }\;\partial\Omega.\end{array}\right.
\]
Then, by Lemma~\ref{subsol}, either $\tilde u>0$ almost everywhere or $\tilde u \leq 0$ in $\Omega$.
Combined with the fact that $\int_{\Omega} \tilde u=0$, this implies that $\tilde u=0$ in $\Omega$.
Therefore, we see that $u=v+c\widehat{u}$, which is a solution to the problem \eqref{eq1311sun}.
\end{proof}

The same is true for the adjoint equation as well.
\begin{proposition}
Let $\Omega$ be a bounded Lipschitz domain.
Assume that $\mathbf A$ satisfies the uniform ellipticity and boundedness condition \eqref{ellipticity} and $\vec c \in L^n(\Omega)$.
Also, let $f\in L^{\frac{2n}{n+2}}(\Omega)$, $\vec F\in L^2(\Omega)$ and $g\in L^{2-\frac{2}{n}}(\Omega)$ satisfy the compatibility condition \eqref{eq:Comp2}.
If $u\in Y_1^2(\Omega)$ is a subsolution to the Neumann problem
\[
\left\{\begin{array}{c l}-\dv(\mathbf A\nabla u)+\vec c \cdot \nabla u= f-\dv \vec F & \text{in }\;\Omega\\
\partial u/\partial \nu = g+\vec F\cdot \nu & \text{on }\;\partial\Omega,\end{array}\right.
\]
then $u$ is, in fact, a solution to the same problem.
\end{proposition}
\begin{proof}
The proof is similar to the proof of Proposition~\ref{SubImplies Sol1}, where we use Proposition~\ref{ExistenceDelta=0Adj} and Lemma~\ref{subsol2}.
\end{proof}

%------------------------------------------------------------------------------%


\begin{thebibliography}{m}
%------------------------------------------------------------------------------%


\bibitem{BS88}
Bennett, Colin; Sharpley, Robert.
\textit{Interpolation of operators}.
Pure and Applied Mathematics, \textbf{129}. 
Academic Press, Inc., Boston, MA, 1988.

\bibitem{BM76}
Bottaro, Gianfranco; Marina, Maria Erminia.
\textit{Problema di Dirichlet per equazioni ellittiche di tipo variazionale su insiemi non limitati}. (Italian)
Boll. Un. Mat. Ital. (4) \textbf{8} (1973), 46--56.


\bibitem{CK13a}
Choi, Jongkeun; Kim, Seick.
\textit{Green's function for second order parabolic systems with Neumann boundary condition}.
J. Differential Equations \textbf{254} (2013), no. 7, 2834--2860.

\bibitem{CK13b}
Choi, Jongkeun; Kim, Seick.
\textit{A Neumann functions for second order elliptic systems with measurable coefficients}.
Trans. Amer. Math. Soc. \textbf{365} (2013), no. 12, 6283--6307.

\bibitem{DS79}
DeVore, R.; Scherer, K.
\textit{Interpolation of linear operators on Sobolev spaces}.
Ann. of Math. (2) \textbf{109} (1979), no. 3, 583--599.


\bibitem{DiB10}
DiBenedetto, E.
\textit{Partial differential equations}.
Second edition. Birkh\"auser Boston, Inc., Boston, MA, 2010.

\bibitem{DV09}
Droniou, J\'er\^ome; V\`azquez, Juan-Luis.
\textit{Noncoercive convection-diffusion elliptic problems with Neumann boundary conditions}.
Calc. Var. Partial Differential Equations \textbf{34} (2009), no. 4, 413--434.


\bibitem{EG15}
Evans, Lawrence C.; Gariepy, Ronald F.
\textit{Measure theory and fine properties of functions}.
Revised edition. Textbooks in Mathematics. CRC Press, Boca Raton, FL, 2015.


\bibitem{GT01}
Gilbarg, David; Trudinger, Neil S.
\textit{Elliptic partial differential equations of second order}.
Reprint of the 1998 edition.
Classics in Mathematics. Springer-Verlag, Berlin, 2001.

\bibitem{Grafakos}
Grafakos, Loukas.
\textit{Classical Fourier analysis}.
Third edition.
Graduate Texts in Mathematics, \textbf{249}. Springer, New York, 2014

\bibitem{HK07}
Hofmann, Steve; Kim, Seick.
\textit{The Green function estimates for strongly elliptic systems of second order}.
Manuscripta Math. \textbf{124} (2007), no. 2, 139--172.

\bibitem{KS11}
Kenig, Carlos; Shen, Zhongwei.
\textit{Layer potential methods for elliptic homogenization problems}.
Comm. Pure Appl. Math. \textbf{64} (2011), no. 1, 1--44.


\bibitem{KS19}
Kim, Seick; Sakellaris, Georgios.
\textit{Green's function for second order elliptic equations with singular lower order coefficients}.
Comm. Partial Differential Equations \textbf{44} (2019), no. 3, 228--270.

 \bibitem{LU68}
Ladyzhenskaya, O. A.; Ural'tseva, N. N. 
\textit{Linear and quasilinear elliptic equations}.
Translated from the Russian by Scripta Technica, Inc.
Academic Press, New York-London, 1968.

\bibitem{NU12}
Nazarov, A. I.; Ural?tseva, N. N.
\textit{The Harnack inequality and related properties of solutions of elliptic and parabolic equations with divergence-free lower-order coefficients}. (Russian) Algebra i Analiz \textbf{23} (2011), no. 1, 136--168; translation in St. Petersburg Math. J. \textbf{23} (2012), no. 1, 93--115 


\bibitem{Mou19}
Mourgoglou, Mihalis.
\textit{Regularity theory and Green's function for elliptic equations with lower order terms in unbounded domains}.
arXiv:1904.04722 [math.AP]

\bibitem{Sak19}
Sakellaris, Georgios.
\textit{Boundary value problems in Lipschitz domains for equations with lower order coefficients}.
 Trans. Amer. Math. Soc. \textbf{372} (2019), no. 8, 5947--5989.

\bibitem{Sak21}
Sakellaris, Georgios.
\textit{On scale-invariant bounds for the Green's function for second-order elliptic equations with lower-order coefficients and applications}.
Anal. PDE \textbf{14} (2021), no. 1, 251--299.

\bibitem{Sak21b}
Sakellaris, Georgios.
\textit{Scale invariant regularity estimates for second order elliptic equations with lower order coefficients in optimal spaces}.
J. Math. Pures Appl. (9) \textbf{156} (2021), 179--214.

\bibitem{Stamp65}
Stampacchia, Guido.
\textit{Le probl\`eme de Dirichlet pour les \'equations elliptiques du second ordre \`a coefficients discontinus}. (French)
Ann. Inst. Fourier (Grenoble) \textbf{15} (1965), fasc. 1, 189--258. 

\bibitem{Ste70}
Stein, Elias M.
\textit{ Singular integrals and differentiability properties of functions}.
Princeton Mathematical Series, No. 30.
Princeton University Press, Princeton, N.J. 1970.


\end{thebibliography}
\end{document}